\documentclass[11pt,a4paper]{article}
\usepackage{multirow}
\usepackage{epsfig}
\usepackage{epstopdf}
\usepackage[T1]{fontenc}
\usepackage{geometry}
\usepackage{amsbsy,amsmath,latexsym,amsfonts, epsfig, color, authblk, amssymb, graphics, bm}
\usepackage{epsf,slidesec,epic,eepic}
\usepackage{fancybox}
\usepackage{fancyhdr}
\usepackage{setspace}
\usepackage{cases}
\usepackage{subfigure}
\usepackage{epsfig}
\usepackage{epstopdf}
\usepackage{nccmath}
\usepackage{algorithm}
\usepackage{algorithmic}
\usepackage[colorlinks, citecolor=blue]{hyperref}

\newtheorem{theorem}{Theorem}[section]
\newtheorem{lemma}{Lemma}[section]
\newtheorem{definition}{Definition}[section]

\newtheorem{proposition}{Proposition}[section]
\newtheorem{corollary}{Corollary}[section]

\newtheorem{remark}{Remark}[section]

\newenvironment{proof}{{\noindent \bf Proof:}}{\hfill$\Box$\medskip}

\definecolor{lred}{rgb}{1,0.8,0.8}
\definecolor{lblue}{rgb}{0.8,0.8,1}
\definecolor{dred}{rgb}{0.6,0,0}
\definecolor{dblue}{rgb}{0,0,0.5}
\definecolor{dgreen}{rgb}{0,0.5,0.5}

\title{Tilt stability of a class of nonlinear semidefinite programs}
\author{Yulan Liu\footnote{(ylliu@gdut.edu.cn), School of Mathematics and Statistics, Guangdong University of Technology, Guangzhou},\ \
Shaohua Pan\footnote{(shhpan@scut.edu.cn), School of Mathematics, South China University of Technology, Guangzhou.}\ \ {\rm and}\ \ Shujun Bi\footnote{Corresponding author:(bishj@scut.edu.cn), School of Mathematics, South China University of Technology, Guangzhou.}}
\date{}

\begin{document}

 \maketitle

 \begin{abstract}
 This paper concerns the tilt stability of local optimal solutions to a class of nonlinear semidefinite programs, which involves a twice continuously differentiable objective function and a convex feasible set. By leveraging the second subderivative of the extended-valued objective function and imposing a suitable restriction on the multiplier, we derive two point-based sufficient characterizations for tilt stability of local optimal solutions around which the objective function has positive semidefinite Hessians, and for a class of linear positive semidefinite cone constraint set, establish a point-based necessary characterization with a certain gap from the sufficient one. For this class of linear positive semidefinite cone constraint case, under a suitable restriction on the set of multipliers, we also establish a point-based sufficient and necessary characterization, which is weaker than the dual constraint nondegeneracy when the set of multipliers is singleton. As far as we know, this is the first work to study point-based sufficient and/or necessary characterizations for nonlinear semidefinite programs with convex constraint sets without constraint nondegeneracy conditions.
\end{abstract}


 \section{Introduction}\label{sec1.0}

 The notion of tilt stability was proposed by Poliquin and Rockafellar \cite{Poliquin98} for the local minimum of an extended real-valued function, which applies to constrained optimization problems via the indicator function of the constraint set. Such stability is a kind of single-valued Lipschitzian behavior of local minimum with respect to one-parametric linear or tilt perturbation, which is basically equivalent to a uniform second-order growth condition as well as strong metric regularity of the subdifferential by \cite[Theorem 3.3]{Drusvyatskiy13} (see also \cite{Mordu131}). Its introduction aims to justify convergence properties, stopping criteria, and the robustness of numerical algorithms. Tilt stability has attracted wide attention in the literature especially in recent years (see, e.g., \cite{Mordu151,Gfrerer15,ChienSiam18,Benko19,Nghia24}).     

 In \cite{Poliquin98}, Poliquin and Rockafellar proved that for the local minimum of an extended real-valued function, under the mild assumptions of prox-regularity and subdifferential continuity, a stationary point is a tilt-stable local minimizer if and only if the second-order limiting subdifferential/generalized Hessian introduced by Mordukhovich \cite{Mordu92} is positive-definite at the point in question. Later, by developing a new second-order subdifferential calculus, Mordukhovich and Rockafellar \cite{Mordu12} derived second-order characterizations of tilt-stable minimizers for some classes of constrained optimization problems. Among others, for $\mathcal{C}^2$-smooth nonlinear programs (NLPs), they showed that under the linear independence constraint qualification (LICQ), a stationary point is a tilt-stable local minimizer if and only if the strong second-order sufficient condition (SSOSC) holds. This result implies that in this setting, if the LICQ at the point in question holds, the tilt stability is equivalent to Robinson's strong regularity \cite{Robinson80} of the associated KKT system. Observe that tilt stability does not postulate LICQ as a necessary condition. This inspires some researchers to delve into tilt stability for NLPs under weaker conditions. 
 
 Mordukhovich and Outrata \cite{Mordu132} proved that SSOSC is sufficient for a stationary point to be a tilt-stable local minimizer of NLPs under the MFCQ and the constant rank constraint qualification (CRCQ). Mordukhovich and Nghia \cite{Mordu151} showed that SSOSC is indeed not a necessary condition for tilt stability and then introduced the uniform second-order sufficient condition to characterize tilt stability when both MFCQ and CRCQ occur. Recently, Gfrerer and Mordukhovich \cite{Gfrerer15} obtained some point-based second-order sufficient conditions for tilt-stable local minimizers of NLPs under some weak conditions including the so-called metric subregularity constraint qualification (MSCQ) and the bounded extreme point property, and established some point-based second-order characterizations of tilt stability under certain additional assumptions. All the aforementioned references employ the uniform positive-definiteness of second-order limiting subdifferentials \cite{Poliquin98,Mordu12,Mordu131} or combined second-order subdifferentials \cite{Mordu151,Gfrerer15}. Different from them, Chieu et al.'s paper \cite{ChienSiam18} adopted the subgradient graphical derivative of an extended real-valued function to characterize tilt stability and extend to the case of NLPs.

 Though the tilt stability of NLPs has received active attention in the past ten years, there are a few works on the tilt stability of non-polyhedral conic programs and most of the existing results were obtained under strong constraint nondegeneracy assumption. Morduhovich et al. \cite{Mordu14} employed the second-order subdifferential calculation from \cite{Outrata08} to present a characterization of tilt-stable minimizers for second-order cone programs (SOCPs) in the extended SSOSC under a nondegeneracy assumption. Similar characterizations involving nondegeneracy and uniqueness of Lagrange multipliers were also established in \cite{Mordu152,Mordu151} for conic programs with a $\mathcal{C}^2$-cone reducible constraint set, where the second-order subdifferential terms are calculated in \cite{Mordu152}. Recently, for the minimization of a sum of a twice continuously differentiable convex function and a proper lower continuous (lsc) convex function, Nghia \cite{Nghia24} presented a novel but abstract characterization for tilt stability by the quadratic growth of the nonsmooth convex function with respect to a certain set, and for the nonsmooth function being the indicator of positive semidefinite (PSD) cone, achieved a specific point-based sufficient and necessary characterization. Note that the constraint nondegeneracy automatically holds for this class of composite optimization. To the best of our knowledge, Benko et al.'s work \cite{Benko19} is the only one to achieve the characterization of tilt stability for non-polyhedral conic programs without constraint nondegeneracy. They obtained neighborhood characterizations and much more difficult point-based characterizations under an appropriate MSCQ for the tilt stability of SOCPs. However, the SOCP considered in \cite{Benko19} only involves a second-order cone, which is essentially equivalent to an NLP.

 In this paper, we are interested in the characterization of tilt-stable local minimizers for nonlinear semidefinite programs (SDPs) with a convex feasible set. Let $\mathbb{S}^n$ denote the space of all $n\times n$ real symmetric matrices, equipped with the trace inner product $\langle \cdot,\cdot\rangle$ and its induced Frobenius norm $\|\cdot\|_F$, and let $\mathbb{S}_{+}^n$ be the cone consisting of all PSD matrices from the space $\mathbb{S}^n$. This class of nonlinear SDPs takes the following form  
 \begin{equation}\label{pNSDP}
  \min_{x\in\mathbb{X}}\Big\{\varphi(x)\ \ {\rm s.t.}\ \ \mathcal{A}x=b,\,g(x)\in\mathbb{S}^n_+\Big\},
 \end{equation}
 where $\varphi\!:\mathbb{X}\to\overline{\mathbb{R}}:=(-\infty,\infty]$ is a proper lsc function that is twice continuously differentiable on an open set $\mathcal{O}\supset\Gamma\!:=\mathcal{A}^{-1}(b)\cap g^{-1}(\mathbb{S}^n_+)$, $\mathcal{A}\!:\mathbb{X}\to\mathbb{R}^m$ and $b\in\mathbb{R}^m$ are the linear mapping and vector, and  $g\!:\mathbb{X}\to\mathbb{S}^n$ is a twice continuously differentiable mapping that will be assumed to be $\mathbb{S}_{-}^n$-convex later. The main contribution of this work is to derive point-based sufficient and/or necessary characterizations for tilt-stable local minimizers of \eqref{pNSDP} without constraint nondegeneracy assumption, which greatly enhances the result of \cite{Nghia24} for a special SDP as well as provides a novel way to study tilt stability of this class of non-polyhedral conic programs. Different from the existing works, we establish our results by operating directly the second subderivative of the extended-valued objective function $\Phi\!:=\varphi+\delta_{\Gamma}$, where $\delta_{\Gamma}$ is the indicator of the feasible set $\Gamma$. In addition, we overcome the difficulty caused by the nonuniqueness of Lagrange multipliers by imposing a restriction on the multiplier skillfully. Specifically, in Section \ref{sec3.1} we derive two point-based sufficient characterizations for tilt-stable local minimizers around which $\nabla^2\varphi(x)$ is positive semidefinite; see Theorems \ref{Scond1-ptilt} and \ref{Scond2-ptilt}. Among others, the characterization of Theorem \ref{Scond1-ptilt} needs the multiplier of minimum rank, while that of Theorem \ref{Scond2-ptilt} imposes a restriction on the set of multiplier sets but does not require it to be a singleton. In Section \ref{sec3.2}, for a class of linear mapping $g$, we obtain a point-based necessary characterization with a certain gap from the converse of the sufficient characterization of Theorem \ref{Scond1-ptilt}. Then, we sharpen this result and establish a point-based sufficient and necessary characterization. In Section \ref{sec3.3}, these characterizations are applied to the standard linear primal and dual SDPs, and are demonstrated to be weaker than the dual constraint nondegeneracy. 
 \section{Notation and preliminaries}\label{sec2}

 Throughout this paper, a hollow capital such as $\mathbb{X},\mathbb{Y}$ and $\mathbb{Z}$ represents a finite dimensional real Euclidean space, equipped with inner product $\langle \cdot,\cdot\rangle$ and its induced norm $\|\cdot\|$, $\mathbb{B}(x,\delta)$ denotes the closed ball of radius $\delta$ centered at $x$, and $\mathbb{O}^n$ denotes the set of all $n\times n$ real orthogonal matrices. For a positive integer $k$, $[k]\!:=\{1,\ldots,k\}$ and  $I_{k}$ denote the $k\times k$ identity matrix. For a matrix $X\in\mathbb{S}^n$, let $X^{\dagger}$ represent the Moore-Penrose pseudo-inverse of $X$, let $\lambda_1(X)\ge\cdots\ge\lambda_n(X)$ denote its eigenvalues arranged in a nonincreasing order, let $\lambda^{+}(X)$ and $\lambda^{-}(X)$ be the subvectors of $\lambda(X)=(\lambda_1(X),\ldots,\lambda_n(X))^{\top}$ obtained by deleting those non-positive components and those nonnegative components of $\lambda(X)$, respectively, and write $\mathbb{O}(X):=\big\{P\in\mathbb{O}^n\,|\,X=P{\rm Diag}(\lambda(X))P^{\top}\big\}$. For any index sets $J_1,J_2\subset[n]$, $0_{J_1J_2}$ denotes a $|J_1|\times|J_2|$ zero matrix, and when $J_1=J_2$, we write $0_{J_1}$ for $0_{J_1J_1}$. For a vector $x\in\mathbb{X}$, $[\![x]\!]$ denotes the subspace spanned by $x$ whose orthogonal complementary is denoted by $[\![x]\!]^\perp$. For a closed set $\Delta\subset\mathbb{X}$, $\delta_{\Delta}$ and $\sigma_{\!\Delta}$ denote the indicator and support functions of $\Delta$, respectively. For a closed convex cone $K\subset\mathbb{X}$, ${\rm lin}(K):=K\cap(-K)$ is the largest subspace contained in $K$, and ${\rm Sp}(K):=\mathbb{R}_{+}(K-K)$ is the linear space generated by the cone $K$. For a linear mapping $\mathcal{B}\!:\mathbb{X}\to\mathbb{Y}$, ${\rm Im}\,\mathcal{B}$ and ${\rm Ker}\,\mathcal{B}$ denote the image and null space of $\mathcal{B}$, respectively, and $\mathcal{B}^*\!:\mathbb{Y}\to\mathbb{X}$ denotes its adjoint. For a twice differentiable mapping $G\!:\mathbb{X}\to\mathbb{Y}$, $\nabla G(x)\!:\mathbb{Y}\to\mathbb{X}$ means the adjoint of $G'(x)\!:\mathbb{X}\to\mathbb{Y}$, the differential mapping of $G$ at $x$, and $D^2G(x)$ denotes its twice differential mapping of $G$ at $x$. For an extended real-valued function $h\!:\mathbb{X}\to\overline{\mathbb{R}}\!:=\!(-\infty,\infty]$, ${\rm dom}\,h\!:=\!\{x\!\in\! \mathbb{X}\,|\, h(x)\!<\!+\infty\}$ represents its domain.
 
 Now we recall the notion of regular and general subdifferentials of an extended real-valued function. For more details, please refer to the excellent monographs \cite{RW98, Mordu18}.
 \begin{definition}
  (see \cite[Definition 8.3]{RW98}) Consider a function $h\!:\mathbb{X}\to\overline{\mathbb{R}}$ and a point $x\in{\rm dom}\,h$. The regular subdifferential of $h$ at $x$ is defined as
  \[
   \widehat{\partial}h(x):=\bigg\{v\in\mathbb{X}\ |\ 
   \liminf_{x'\to x\atop x'\ne x}\frac{h(x')-h(x)-\langle v,x'-x\rangle}{\|x'-x\|}\ge 0\bigg\};
  \]
  the general (known as limiting or Morduhovich) subdifferential of $h$ at $x$ is defined as
  \[
   \partial h(x):=\Big\{v\in\mathbb{X}\ |\ \exists\,x^k\to x\ {\rm with}\ h(x^k)\to h(x)\ {\rm and}\ v^k\in\widehat{\partial}h(x^k)\ {\rm s.t.}\ v^k\to v\Big\}.
 \]
 \end{definition}
 
 When $h=\delta_{\Omega}$ for a nonempty closed set $\Omega\subset\mathbb{X}$, its regular subdifferential $\partial h(x)$ is the regular normal cone to $\Omega$ at $x$, denoted by $\widehat{\mathcal{N}}_{\Omega}(x)$, and its subdifferential reduces to the normal cone to $\Omega$ at $x$, denoted by $\mathcal{N}_{\Omega}(x)$. When $\Omega$ is convex, at any $x\in\Omega$, they reduce to 
 the one in the sense of convex analysis, i.e., $\big\{v\in\mathbb{X}\ |\ \langle v,x'-x\rangle\ge 0\ \ \forall x'\in\Omega\big\}$. 
 \subsection{Tangent cone and second-order tangent set}\label{sec2.1}

 Before introducing tangent cone to a set, we recall the notion of subderivative function.
 \begin{definition}\label{sderive-def}
  (see \cite[Definitions 8.1 $\&$ 7.20]{RW98}) Consider a function $h\!:\mathbb{X}\to\mathbb{\overline{R}}$ and a point $x\in{\rm dom}\,h$. Denote the first-order difference quotients of $h$ at $x$ by
  \[
   \Delta_{\tau}h(x)(w'):=\tau^{-1}\big[h(x+\tau w')-h(x)\big]\quad{\rm for}\ \tau>0.
  \] 
  The subderivative function of $h$ at $x$, denoted by $dh(x)\!:\mathbb{X}\to[-\infty,\infty]$, is defined as
  \[
   dh(x)(w):=\liminf_{\tau\downarrow 0\atop w'\to w}\Delta_{\tau}h(x)(w')\quad{\rm for}\ w\in \mathbb{X}.
  \]
 \end{definition}
 
 With the subderivative function, we follow the same line as in \cite{MohMSSiam20} to introduce the critical cone of an extended real-valued function $h\!:\mathbb{X}\to\overline{\mathbb{R}}$ at a point pair $(x,v)\in{\rm gph}\,\partial h$:
 \begin{equation*}
 \mathcal{C}_{h}(x,v):=\big\{w\in\mathbb{X}\ |\ dh(x)(w)=\langle v,w\rangle\big\}.
 \end{equation*}  
 When $h=\delta_{\Omega}$ for a nonempty closed set $\Omega\subset\mathbb{X}$, the subderivative $dh(x)$ for $x\in{\rm dom}\,h$ is the indicator of $\mathcal{T}_{\Omega}(x)$, the tangent cone to $\Omega$ at $x$, and now $\mathcal{C}_{h}(x,v)=\mathcal{T}_{\Omega}(x)\cap[\![v]\!]^{\perp}:=\mathcal{C}_{\Omega}(x,v)$. 
 The tangent and inner tangent cones to $\Omega$ at $x\in\Omega$ are respectively defined as
 \begin{align*}
 \mathcal{T}_{\Omega}(x)&:=\big\{w\in\mathbb{X}\ |\ \exists\, \tau_k\downarrow 0\ \ {\rm s.t.}\ \ {\rm dist}(x+\tau_kw,\Omega)=o(\tau_k)\big\},\\
 \mathcal{T}_{\Omega}^{i}(x)&:=\big\{w\in\mathbb{X}\ |\ {\rm dist}(x+\tau w,\Omega)=o(\tau)\ \ \forall\tau\ge 0\big\}.
 \end{align*}
 We stipulate that $\mathcal{T}_{\Omega}(x)\!=\!\mathcal{T}^i_{\Omega}(x)=\emptyset$ when $x\notin\Omega$. The second-order tangent set to $\Omega$ at $x\in\Omega$ for $w\in\mathcal{T}_{\Omega}(x)$ is defined as
 \[
  \mathcal{T}^{2}_{\Omega}(x,w)
  :=\Big\{z\in\mathbb{X}\ |\ \exists\, \tau_k\downarrow 0\ \ {\rm s.t.}\ \ {\rm dist}(x+\tau_kw+(\tau_k^2/2)z,\Omega)=o(\tau_k^2)\Big\},
 \]
 while the inner second-order tangent set to $\Omega$ at $x\in \Omega$ for $w\in\mathcal{T}_{\Omega}^{i}(x)$ is defined as
 \[
 \mathcal{T}^{i,2}_{\Omega}(x,w)
 :=\Big\{z\in\mathbb{X}\ |\ {\rm dist}(x+\tau w+(\tau^2/2)z,\Omega)=o(\tau^2)\quad \forall\tau\ge 0\Big\}.
 \]
 One can check that $\mathcal{T}^{2}_{\Omega}(x,w)=\emptyset$ if $w\notin\mathcal{T}_{\Omega}(x)$, and $\mathcal{T}^{i,2}_{\Omega}(x,w)=\emptyset$ if  $w\notin\mathcal{T}_{\Omega}^{i}(x)$. The set $\Omega$ is called parabolically derivable at $x$ for $w\in\mathcal{T}_{\Omega}(x)$ if $\mathcal{T}^{i,2}_{\Omega}(x,w) =\mathcal{T}^{2}_{\Omega}(x,w)\ne\emptyset$.

 Next we recall the tangent and normal cones to $\mathbb{S}_{+}^n$ at any given $X\!\in\mathbb{S}_{+}^n$. Pick any $S\in\mathcal{N}_{\mathbb{S}_{+}^n}(X)$. Then, $X$ and $S$ have a simultaneous ordered eigenvalue decomposition
 \begin{equation}\label{XSeig}
  X=P\begin{pmatrix}
   {\rm Diag}(\lambda^{+}(X)) & 0\\
   0 & 0\\
  \end{pmatrix}P^\top\ \ {\rm and}\ \ 
  S=P\begin{pmatrix}
   0 & 0 \\
   0 & {\rm Diag}(\lambda^{-}(S))
   \end{pmatrix}P^\top 
 \end{equation}
 with $P\!\in\mathbb{O}(X)\cap\mathbb{O}(S)$. Let $P_{0}$ (resp. $P_{+}$ and $P_{-}$) be the matrix obtained by removing those columns of $P$ whose indices do not belong to $\{i\in[n]\,|\,\lambda_i(X)=0=\lambda_i(S)\}$  (resp. $\{i\in[n]\,|\,\lambda_i(X)>0\}$ and $\{i\in[n]\,|\,\lambda_i(S)<0\}$). By \cite[Section 2]{Sun06} or \cite[Section 5.3.1]{BS00}, 
 \begin{align}\label{SDCTan}  
  \mathcal{T}_{\mathbb{S}^n_+}(X)=\Big\{W\in \mathbb{S}^n\ |\ [P_{0}\ \ P_{-}]^{\top}W[P_{0}\ \ P_{-}]\succeq 0\Big\},\qquad\quad\\
  \label{SDCNor}
  \mathcal{N}_{\mathbb{S}^n_+}(X)=\Big\{W\in \mathbb{S}^n\ |\ [P_{0}\ \ P_{-}]^{\top}W[P_{0}\ \ P_{-}]\preceq 0,\,P_{+}^{\top} W P=0\Big\},
 \end{align}
 while the  critical cone of $\mathbb{S}^n_+$ at $X$ associated with $S$ takes the following form
 \begin{equation}\label{SDCritial}
 \mathcal{C}_{\mathbb{S}^n_+}(X,S)=\Big\{W\in \mathbb{S}^n\,|\,  P_{0}^{\top} W P_{0}\succeq 0,\,P_{0}^{\top} W P_{-}=0\ {\rm and}\ P_{-}^{\top} W P_{-}=0\Big\}.
 \end{equation} 
 The following lemma characterizes the support of the second-order tangent set of $\mathbb{S}^n_+$. 
\begin{lemma}\label{sigma-Lemma}
 Consider any $X\in\mathbb{S}_{+}^n$ and $S\in\mathcal{N}_{\mathbb{S}_{+}^n}(X)$. Let $X$ and $S$ have the eigenvalue decomposition as in \eqref{XSeig} with $\alpha=\{i\in[n]\,|\,\lambda_i(X)>0\}$ and $\gamma=\{i\in[n]\,|\,\lambda_i(S)<0\}$. Then, for any $W\in\mathcal{C}_{\mathbb{S}_{+}^n}(X,S)$, with $\widetilde{W}:=P^{\top}WP$ there holds 
 \begin{align*}
 \sigma_{\mathcal{T}^2_{\mathbb{S}^n_+}(X,W)}(S)=2\langle S, WX^{\dagger}W\rangle=
 \left\{\begin{matrix}
     &2\sum\limits_{i\in \alpha,j\in \gamma}\frac{\lambda_j(S)}{\lambda_i(X)}(\widetilde{W}_{ij})^2 &{\rm if}\ \  X\neq 0,\\
     & 0 & {\rm otherwise}.
     \end{matrix}
 \right.
\end{align*}
\end{lemma} 
\begin{proof}
 Fix any $W\in\mathcal{C}_{\mathbb{S}_{+}^n}(X,S)=\mathcal{T}_{\mathbb{S}_{+}^n}(X)\cap[\![S]\!]^{\perp}$. Invoking \cite[Proposition 3.33]{BS00} with $K=\mathbb{S}_{-}^n$ and $\mathbb{S}^n\ni Z\mapsto G(Z):=-Z$ leads to $\mathcal{T}_{\mathbb{S}_{+}^n}^2(X,W)=-\mathcal{T}_{\mathbb{S}_{-}^n}^2(-X,-W)$. Then, 
  \begin{align*}
  \sigma_{\mathcal{T}_{\mathbb{S}_{+}^n}^2(X,W)}(S)
  &=\sup_{Z\in\mathcal{T}_{\mathbb{S}_{+}^n}^2(X,W)}\langle -Z,-S\rangle=\sup_{Z'\in\mathcal{T}_{\mathbb{S}_{-}^n}^2(-X,-W)}\langle Z',-S\rangle\\
  &=\sigma_{\mathcal{T}_{\mathbb{S}_{-}^n}^2(-X,-W)}(-S)=2\langle S, WX^{\dagger}W\rangle\\
  &=2\langle P^{\top}SP,\widetilde{W}P^{\top}X^{\dagger}P\widetilde{W}\rangle=2\sum\limits_{i\in \alpha,j\in \gamma}\frac{\lambda_j(S)}{\lambda_i(X)}(\widetilde{W}_{ij})^2,
 \end{align*}
 where the sixth equality is obtained by using equation \eqref{XSeig}. 
 \end{proof}
 \subsection{Second subderivative}\label{sec2.2} 

 To introduce the second subderivative of an extended real-valued function $h\!:\mathbb{X}\to\mathbb{\overline{R}}$, we define its 
 second-order difference quotient at a point $x$ for a vector $v$ by 
 \begin{equation*}
 \Delta^2_{\tau}h(x|v)(w')
  :=\frac{h(x\!+\!\tau w')-h(x)-\tau\langle v,w'\rangle}{\tau^2/2}\quad{\rm for}\ \tau>0.
 \end{equation*} 
 \begin{definition}\label{Def-Ssubderiv}
 (\cite[Definitions 13.3 $\&$ 13.6]{RW98}) Consider a proper $h\!:\mathbb{X}\to\mathbb{\overline{R}}$, a point $x\in{\rm dom}\,h$ and a vector $v\in\mathbb{X}$. The second subderivative of $h$ at $x$ for $v$ is defined as
 \[
   d^2h(x|v)(w):=\liminf_{\tau\downarrow 0\atop w'\to w}\Delta^2_{\tau}h(x|v)(w')\quad\ \forall w\,\in\mathbb{X}.
 \]
 \end{definition}

 By \cite[Proposition 13.5]{RW98}, the function $d^2h(x|v)$ is lsc and positively homogenous of degree $2$; and if it is proper, the inclusion ${\rm dom}\,d^2h(x|v)\!\subset\!\big\{w\!\in\!\mathbb{X}\,|\,dh(x)(w)\!=\!\langle v,w\rangle\big\}$ holds.

\begin{lemma}\label{d2Ab-lemma}
 Consider any $X\!\in\mathcal{A}^{-1}(b)$ and $S\in\mathcal{N}_{\mathcal{A}^{-1}(b)}(X)$. Then, for any $W\in\mathbb{S}^n$, 
 \[
   d^2\delta_{\mathcal{A}^{-1}(b)}(X|S)(W)=\delta_{{\rm Ker}\mathcal{A}}(W).
 \]
\end{lemma}
\begin{proof}
 Fix any $W\in\mathbb{S}^n$. The polyhedrality of $\mathcal{A}^{-1}(b)$ implies that it is parabolically derivable at $X$. Invoking \cite[Theorem 3.3 (i)]{Moha2020} with $\Omega=\mathcal{A}^{-1}(b)$ results in 
 \[
  {\rm dom}\,d^2\delta_{\mathcal{A}^{-1}(b)}(X|S)=\mathcal{C}_{\mathcal{A}^{-1}(b)}(X,S)={\rm Ker}\mathcal{A}\cap [\![S]\!]^\perp. 
 \]
 As $S\in\mathcal{N}_{\mathcal{A}^{-1}(b)}(X)={\rm Im}\mathcal{A}^*$, we have ${\rm Ker}\mathcal{A}\cap [\![S]\!]^\perp={\rm Ker}\mathcal{A}$. Together with the above equation, when $W\notin{\rm Ker}\mathcal{A}$, $d^2\delta_{\mathcal{A}^{-1}(b)}(X|S)(W)=\infty$. When $W\!\in {\rm Ker}\mathcal{A}$, applying \cite[Theorem 5.6 \& Eq. (5.2)]{Moha2020} with $f(X)=\mathcal{A}X-b$ and $\Theta=\{0\}^m$ yields that
 \begin{align*}
 d^2\delta_{\mathcal{A}^{-1}(b)}(X|S)(W)&=\max_{y\in(\mathcal{A}^*)^{-1}(S)}d^2 \delta_{\Theta}(\mathcal{A}X\!-\!b\,|\,y)(\mathcal{A}W)\\
 &=\max_{y\in(\mathcal{A}^*)^{-1}(S)}{-\sigma_{\mathcal{T}^2_{\Theta}(\mathcal{A}X-b,\mathcal{A}W)}(y)}=0,
\end{align*}
 where the second equality is by \cite[Theorem 3.3]{Moha2020}, and the third is implied by $\mathcal{T}^2_{\Theta}(\mathcal{A}X\!-\!b,\mathcal{A}W)=\{0\}$ because $\Theta=\{0\}^m$ and $\mathcal{A}W=0$. By the arbitrariness of $W\in\mathbb{S}^n$, we get the desired equality. 
 \end{proof}
\subsection{Metric (sub)regularity of multifunctions}\label{sec2.3}

 Let $\mathcal{F}\!:\mathbb{X}\rightrightarrows\mathbb{Z}$ be a multifunction. Denote its domain by ${\rm dom}\,\mathcal{F}:=\{x\in\mathbb{X}\,|\,\mathcal{F}(x)\ne\emptyset\}$. The mapping $\mathcal{F}$ is closed if its graph, denoted by ${\rm gph}\,\mathcal{F}\!:=\big\{(x,z)\in\mathbb{X}\times\mathbb{Z}\,|\,z\in\mathcal{F}(x)\big\}$, is a closed set in $\mathbb{X}\times\mathbb{Z}$. We say that the mapping $\mathcal{F}$ is outer semicontinuous (osc) at $x\in\mathbb{X}$ if $\limsup_{x'\to x}\mathcal{F}(x')\subset\mathcal{F}(x)$, and say that $\mathcal{F}$ is osc if it is osc everywhere. 
\begin{definition}\label{subregular}
 A multifunction $\mathcal{F}\!:\mathbb{X}\rightrightarrows\mathbb{Z}$ is said to be (metrically) regular at $(x,z)\in{\rm gph}\,\mathcal{F}$ if there exist $\delta>0$ and $\kappa>0$ such that for all $(x',z')\in\mathbb{B}((x,z),\delta)$, 
 \[
 {\rm dist}(x',\mathcal{F}^{-1}(z'))\le\kappa\,{\rm dist}(z',\mathcal{F}(x')), 
 \]
 while the multifunction $\mathcal{F}$ is said to be (metrically) subregular at a point $(x,z)\in{\rm gph}\,\mathcal{F}$ if there exist $\delta>0$ and $\kappa>0$ such that for all $x'\in\mathbb{B}(x,\delta)$, 
 \[
  {\rm dist}(x',\mathcal{F}^{-1}(z))\le\kappa\,{\rm dist}(z,\mathcal{F}(x')).
 \]
 \end{definition}
 \begin{remark}\label{remark1-regular}
 {\bf(a)} The metric regularity of $\mathcal{F}$ is robust in the sense that if $\mathcal{F}$ is metrically regular at $(\overline{x},\overline{y})\in{\rm gph}\mathcal{F}$, then there exists a neighborhood $\mathcal{Z}$ of $(\overline{x},\overline{y})$ such that the mapping $\mathcal{F}$ is metrically regular at any $(x,y)\in\mathcal{Z}$ with $y\in\mathcal{F}(x)$. 
 
 \noindent
 {\bf(b)} The metric subregularity of $\mathcal{F}$ is robust in the sense that if $\mathcal{F}$ is metrically subregular at $(\overline{x},\overline{y})\in{\rm gph}\mathcal{F}$, then there exists a neighborhood $\mathcal{U}$ of $\overline{x}$ such that the mapping $\mathcal{F}$ is metrically subregular at any $x\in\mathcal{U}$ with $\overline{y}\in\mathcal{F}(x)$. 
 \end{remark}

 For the metric regularity of constraint systems, we have the following conclusion. 
 \begin{lemma}\label{regular-approx}
  Let $G\!:\mathbb{X}\to\mathbb{Y}$ be a continuously differentiable mapping, and let $K\subset\mathbb{Y}$ be a closed convex set. Consider the multifunction $\mathcal{F}(x)\!:=G(x)-K$ for $x\in\mathbb{X}$. If the mapping $\mathcal{F}$ is metrically regular at $(\overline{x},0)\in{\rm gph}\,\mathcal{F}$, then for any given $d\in\mathbb{Y}$, the mapping $\widetilde{\mathcal{F}}_{\overline{x},d}(u)\!:=G'(\overline{x})u-d-\mathcal{T}_{K}(G(\overline{x}))$ is metrically regular at $(\overline{u},0)\in{\rm gph}\,\widetilde{\mathcal{F}}_{\overline{x},d}$. 
 \end{lemma}
 \begin{proof}
  By \cite[Proposition 2.89 $\&$ 2.97]{BS00}, the metric regularity of $\widetilde{\mathcal{F}}_{\overline{x},d}$ at $(\overline{u},0)$ is equivalent to requiring that
  \begin{equation}\label{aim-regular}
   G'(\overline{x})\mathbb{X}-\mathcal{T}_{\mathcal{T}_{K}(G(\overline{x}))}(G'(\overline{x})\overline{u}-d)=\mathbb{Y}; 
  \end{equation}
  while the metric regularity of the mapping $\mathcal{F}$ at $(\overline{x},0)$ is equivalent to the condition 
  \begin{equation}\label{temp-regular}
    G'(\overline{x})\mathbb{X}-\mathcal{T}_{K}(G(\overline{x}))=\mathbb{Y}. 
  \end{equation}
 From \cite[Example 2.62]{BS00}, we have $\mathcal{T}_{\mathcal{T}_{K}(G(\overline{x}))}(G'(\overline{x})\overline{u}\!-\!d)={\rm cl}(\mathcal{T}_{K}(G(\overline{x}))+[\![G'(\overline{x})\overline{u}\!-\!d]\!])$, which means that $\mathcal{T}_{K}(G(\overline{x}))\subset\mathcal{T}_{\mathcal{T}_{K}(G(\overline{x}))}(G'(\overline{x})\overline{u}\!-\!d)$. Thus, equality \eqref{temp-regular} implies \eqref{aim-regular}. \end{proof}
\begin{remark}\label{remark2-regular}
{\bf(a)} Consider any $x\in\Gamma$. For any given $\widetilde{b}\in\mathbb{R}^m$ and $\widetilde{C}\in\mathbb{S}^n$, define the multifunctions $\mathcal{G}\!:\mathbb{X}\rightrightarrows\mathbb{R}^m\times\mathbb{S}^n$ and $\mathcal{G}_{x}\!:\mathbb{X}\rightrightarrows\mathbb{R}^m\times\mathbb{S}^n$ by
 \begin{align}\label{Gmap}
 \mathcal{G}(z):=\begin{pmatrix}
  \mathcal{A}z-b\\ g(z)
 \end{pmatrix}-\begin{pmatrix}
   \{0\}^m\\ \mathbb{S}_{+}^n
   \end{pmatrix}\ {\rm and}\    
 \widetilde{\mathcal{G}}_{x}(z):=\begin{pmatrix}
  \mathcal{A}z-\widetilde{b}\\ g'(x)z-\widetilde{C}
 \end{pmatrix} -\begin{pmatrix}
   \{0\}^m\\ \mathcal{T}_{\mathbb{S}_{+}^n}(g(x))
 \end{pmatrix}.   
\end{align}
 By Lemma \ref{regular-approx}, if the mapping $\mathcal{G}$ is metrically regular at $(x,(0,0))$, so is the mapping $\widetilde{\mathcal{G}}_x$ at any $(z,(0,0))\in{\rm gph}\,\widetilde{\mathcal{G}}_x$. By \cite[Proposition 2.106]{RW98}, when the following multifunction 
 \begin{equation}\label{Hmap}
  \mathcal{H}(z)\!:=g(z)-\mathbb{S}_{+}^n\quad{\rm for}\ z\in\mathbb{X}
 \end{equation}
 is convex, the mapping $\mathcal{G}$ is metrically regular at every $x\in\Gamma$ if and only if the Slater's CQ holds for problem \eqref{pNSDP}, i.e., there exists $\widehat{x}\in\mathbb{X}$ such that $\mathcal{A}\widehat{x}=b$ and $g(\widehat{x})\in{\rm int}(\mathbb{S}_{+}^n)$. 

\noindent
 {\bf(b)} By invoking \cite[Proposition 2.89 $\&$ 2.97]{BS00}, the metric regularity of $\mathcal{G}$ at $(\overline{x},(0,0))$ implies that of the multifunction $\mathcal{H}$ at $(\overline{x},0)$, so the metric subregularity of $\mathcal{H}$ at $(\overline{x},0)$. 
 
 \end{remark} 

 To close this section, we take a closer look at the local boundedness of multifunction
 \begin{equation}\label{Mmap}
 \!\mathcal{M}(x,v)\!:=\left\{\begin{array}{cl}
  \!\!\big\{(y,S)\in\mathbb{R}^m\times\mathcal{N}_{\mathbb{S}_{+}^n}(g(x))\ |\ \mathcal{A}^*y+\!\nabla g(x)S=v\big\}&{\rm if}\ (x,v)\in{\rm gph}\mathcal{N}_{\Gamma},\\
   \emptyset &{\rm otherwise}.
 \end{array}\right.   
 \end{equation}
 Recall that a multifunction $\mathcal{F}\!:\mathbb{X}\rightrightarrows\mathbb{Z}$ is locally bounded at a point $x\in\mathbb{X}$ if for some neighorhood of $\mathcal{V}$ of $x$, the set
 $\mathcal{F}(\mathcal{V}):=\bigcup_{x\in\mathcal{V}}\mathcal{F}(x)$ is bounded.
 \begin{proposition}\label{Mmap-prop}
  The multifunction $\mathcal{M}\!:\mathbb{X}\times\mathbb{X}\rightrightarrows\mathbb{R}^m\times\mathbb{S}^n$ defined in \eqref{Mmap} is closed. If the mapping $\mathcal{G}$ defined in \eqref{Gmap} is metrically regular at $(x,(0,0))$, then for all $v\in\mathcal{N}_{\Gamma}(x)$ the mapping $\mathcal{M}$ is locally bounded at $(x,v)$; conversely, if the mapping $\mathcal{M}$ is locally bounded at $(x,0)$, then the mapping $\mathcal{G}$ is metrically regular at $(x,(0,0))$.
 \end{proposition}
\begin{proof}
 Fix any $(x,v)\in\mathbb{X}\times\mathbb{X}$. To prove that $\mathcal{M}$ is closed at $(x,v)$, it suffices to argue that $\limsup_{(x',v')\to(x,v)}\mathcal{M}(x',v')\subset\mathcal{M}(x,v)$. If $\mathcal{M}(x,v)=\emptyset$, then $(x,v)\notin{\rm gph}\,\mathcal{N}_{\Gamma}$. From the closedness of ${\rm gph}\mathcal{N}_{\Gamma}$ or equivalently the outer semicontinuity of $\mathcal{N}_{\Gamma}$, there exists $\delta>0$ such that for any $(x',v')\in\mathbb{B}((x,v),\delta)$, $(x',v')\notin{\rm gph}\,\mathcal{N}_{\Gamma}$ and then $\mathcal{M}(x',v')=\emptyset$. This implies that $\limsup_{(x',v')\to(x,v)}\mathcal{M}(x',v')\subset\mathcal{M}(x,v)$. Next we consider that $\mathcal{M}(x,v)\ne\emptyset$. Pick any $(y,S)\in\limsup_{(x',v')\to(x,v)}\mathcal{M}(x',v')$. Then, there exist $(x^k,v^k)\to (x,v)$ and $(y^k,S^k)\to(y,S)$ with $(y^k,S^k)\in\mathcal{M}(x^k,v^k)$ for all $k$. By the definition of $\mathcal{M}$, for any $k$,
 \[
  (x^k,v^k)\in{\rm gph}\,\mathcal{N}_{\Gamma}\ \ {\rm and}\ \ (y^k,S^k)\in\mathbb{R}^m\times\mathcal{N}_{\mathbb{S}_{+}^n}(g(x^k))\ {\rm with}\ \mathcal{A}^*y^k\!+\!\nabla g(x^k)S^k=v^k.
 \]
 The outer semicontinuity of normal cone mappings $\mathcal{N}_{\Gamma}$ and $\mathcal{N}_{\mathbb{S}_{+}^n}$ implies that $(x,v)\in{\rm gph}\,\mathcal{N}_{\Gamma}$ and $S\in\mathcal{N}_{\mathbb{S}_{+}^n}(g(x))$ with $\mathcal{A}^*y\!+\nabla g(x)S=v$. Thus, $(y,S)\in\mathcal{M}(x,v)$. Therefore, $\limsup_{(x',v')\to(x,v)}\mathcal{M}(x',v')\subset\mathcal{M}(x,v)$, and the closedness of $\mathcal{M}$ follows.

 Now assume that $\mathcal{G}$ is metrically regular at $(x,(0,0))$, which by \cite[Propositions 2.89 \& 2.97]{BS00} is equivalent to saying that the following implication holds:
 \begin{equation}\label{imply1}
 \mathcal{A}^*u+\nabla g(x)H=0,\,u\in\mathbb{R}^m,\,H\in\mathcal{N}_{\mathbb{S}_{+}^n}(g(x))\ \Longrightarrow\ (u,H)=0.
 \end{equation}
 Fix any $v\!\in\mathcal{N}_{\Gamma}(x)$. Suppose on the contrary that $\mathcal{M}$ is not locally bounded at $(x,v)$. There exist a sequence $\{(x^k,v^k)\}_{k\in\mathbb{N}}$ with $(x^k,v^k)\to (x,v)$ as $k\to\infty$ and an unbounded sequence $\{(y^k,S^k)\}_{k\in\mathcal{K}}$ with $(y^k,S^k)\in\mathcal{M}(x^k,v^k)$ for each $k\in\mathcal{K}$, where $\mathcal{K}\subset\mathbb{N}$ is an infinite index set. Clearly, for each $k\in\mathcal{K}$, $S^k\in\mathcal{N}_{\mathbb{S}_{+}^n}(g(x^k))$ and $\mathcal{A}^*y^k\!+\nabla g(x^k)S^k=v^k$. For each $k\in\mathcal{K}$, write $\widetilde{S}^k\!:=\frac{S^k}{\sqrt{\|S^k\|_F^2+\|y^k\|^2}},\widetilde{y}^k\!:=\frac{y^k}{\sqrt{\|S^k\|_F^2+\|y^k\|^2}}$ and $\widetilde{v}^k\!:=\frac{v^k}{\sqrt{\|S^k\|_F^2+\|y^k\|^2}}$. Then,  $\widetilde{S}^k\in\mathcal{N}_{\mathbb{S}_{+}^n}(g(x^k))$ and $\mathcal{A}^*\widetilde{y}^k\!+\!\nabla g(x^k)\widetilde{S}^k=\widetilde{v}^k$ for each $k\in\mathcal{K}$. 
 Note that $\{\widetilde{S}^k\}_{k\in\mathbb{N}}$ and $\{\widetilde{y}^k\}_{k\in\mathbb{N}}$ are bounded. If necessary by taking a subsequence, we can assume that $\lim_{\mathcal{K}\ni k\to\infty}\widetilde{S}^k=\widetilde{S}$ and $\lim_{\mathcal{K}\ni k\to\infty}\widetilde{y}^k=\widetilde{y}$ with $\|\widetilde{S}\|_F^2+\|\widetilde{y}\|^2=1$. The outer semicontinuity of $\mathcal{N}_{\mathbb{S}_{+}^n}$ implies that $\widetilde{S}\in\mathcal{N}_{\mathbb{S}_{+}^n}(g(x))$. Furthermore, passing the limit $\mathcal{K}\ni k\to\infty$ to $\mathcal{A}^*\widetilde{y}^k\!+\!\nabla g(x^k)\widetilde{S}^k=\widetilde{v}^k$ and using the boundedness of $\{v^k\}_{k\in\mathbb{N}}$ and the unboundedness of $\big\{\sqrt{\|S^k\|_F^2+\|y^k\|^2}\big\}_{k\in\mathcal{K}}$ leads to $\mathcal{A}^*\widetilde{y}+\nabla g(x)\widetilde{S}=0$. Thus, $\mathcal{A}^*\widetilde{y}+\nabla g(x)\widetilde{S}=0$ with $\widetilde{S}\in\mathcal{N}_{\mathbb{S}_{+}^n}(g(x))$ and $\widetilde{y}\in\mathbb{R}^m$ satisfying  $\|\widetilde{S}\|_F^2+\|\widetilde{y}\|^2=1$, a contradiction to \eqref{imply1}.  

 Conversely, it suffices to establish the implication \eqref{imply1}. Suppose on the contrary that it does not hold. There exists a nonzero $(u,H)\in\mathbb{R}^m\times\mathbb{S}^n$ such that $H\in\mathcal{N}_{\mathbb{S}_{+}^n}(g(x))$ with $\mathcal{A}^*u+\nabla g(x)H=0$.
 Clearly, $(u,H)\in\mathcal{M}(x,0)$, and for every $t>0$, $t(u,H)\in\mathcal{M}(x,0)$. This yields a contradiction to the local boundedness of $\mathcal{M}$ at $(x,0)$. 
 \end{proof}
\subsection{Tilt stability of nonconvex composite optimization}\label{sec2.4}

 We first recall from the work \cite{Poliquin98} the formal definition of tilt-stable local minimum.
\begin{definition}\label{Def-tilt}
 For a proper function $f\!:\mathbb{X}\to\overline{\mathbb{R}}$, a point $\overline{x}\in{\rm dom}\,f$ is said to be a tilt-stable local minimum of $f$ if there exists $\delta>0$ such that the solution mapping
 \[
  M_{\delta}(v):=\mathop{\arg\min}_{x\in\mathbb{B}(\overline{x},\delta)}\big\{f(x)-f(\overline{x})-\langle v,x-\overline{x}\rangle\big\}
 \]
 is single-valued and Lipschitz continuous on some neighborhood of $v\!=\!0$ with $M_{\delta}(0)\!=\!\{\overline{x}\}$.   	
 \end{definition}
  
 As mentioned in the introduction, Nghia \cite{Nghia24} used the second subderivative to provide a point-based characterization for the tilt stability of the minimization of the sum of a twice continuously differentiable convex function and a proper lsc convex function. Here, we are interested in the tilt stability of the nonconvex composite optimization
 \begin{equation}\label{composite-prob}
  \min_{x\in\mathbb{X}}\Theta(x):=\vartheta(x)+h(x),
 \end{equation}
 where $h\!:\mathbb{X}\to\overline{\mathbb{R}}$ is a proper lsc convex function, and $\vartheta\!:\mathbb{X}\to\mathbb{R}$ is a proper lsc function that is twice continuously differentiable on an open set $\mathcal{O}'\supset{\rm dom}\,h$. From Definition \ref{Def-Ssubderiv} and the twice continuous differentiablity of $\vartheta$ on the set $\mathcal{O}'$, the second subderivative of $\Theta$ at any $x\in{\rm dom}\,h$ for $v\in\partial\Theta(x)$ has the following form
 \begin{equation}\label{sderiv-Phi}
  d^2\Theta(x|v)(w)=\langle \nabla^2\vartheta(x)(w),w\rangle+d^2h(x\,|\,v\!-\!\nabla\vartheta(x))(w)\quad\forall w\in \mathbb{X}.
 \end{equation}  
 By combining equation \eqref{sderiv-Phi} with \cite[Lemma 2.4]{Nghia24} and \cite[Theorem 3.3]{ChienSiam18}, we can obtain the following characterization for the tilt-stable local minimizers of \eqref{composite-prob}. Since its proof can be found in \cite[Proposition 2.2]{LiuPan24}, we here omit it. 
 \begin{lemma}\label{Tiltlemma}
  Consider a local optimal solution $\overline{x}$ of \eqref{composite-prob}. The solution $\overline{x}$ is tilt-stable if $\nabla^2\vartheta(x)\succeq 0$ on an open neighborhood $\mathcal{U}\subset\mathcal{O}'$ of $\overline{x}$ and ${\rm Ker}\nabla^2\vartheta(\overline{x})\cap \mathcal{W}=\{0\}$ with
 \begin{align*}
 \mathcal{W}&=\Big\{w\in \mathbb{X}\ |\ \exists\,\{((x^k,y^k),w^k)\}\subset {\rm gph}\,\partial h\times\mathbb{X}\ {\rm with}\ \lim_{k\to\infty} (x^k,y^k,w^k)= (\overline{x},-\nabla \vartheta(\overline{x}),w)\\
 &\qquad\qquad\qquad\ {\rm such\ that}\ \lim_{k\to \infty}d^2 h(x^k|y^k)(w^k)=0\Big\}.
 \end{align*}
 Conversely, if the solution $\overline{x}$ is tilt-stable, then ${\rm Ker}\,\nabla^2\vartheta(\overline{x})\cap \mathcal{W}=\{0\}$. 
 \end{lemma}
 \section{Tilt stability of nonlinear SDPs}\label{sec3}

 Before investigating the tilt stability of problem \eqref{pNSDP}, we provide the first-order varational characterization of its feasible set $\Gamma$ and the expression of the second subderivative of $\delta_{\Gamma}$.    
\begin{proposition}\label{Gamset-prop}
 Consider any $x\in\Gamma$. Suppose that the mapping $\mathcal{G}$ in \eqref{Gmap} is subregular at $(x,(0,0))$, and that the mapping $\mathcal{H}$ in \eqref{Hmap} is subregular at $(x,0)$. Then 
\begin{itemize}
 \item [(i)] $\mathcal{N}_{\Gamma}(x)=\mathcal{N}_{g^{-1}(\mathbb{S}^n_+)}(x)+ \mathcal{N}_{\mathcal{A}^{-1}(b)}(x)=\nabla g(x)\mathcal{N}_{\mathbb{S}^n_+}(g(x))+{\rm Im}\mathcal{A}^{*}$;

 \item [(ii)] $\mathcal{T}_{\Gamma}(x)=\mathcal{T}_{g^{-1}(\mathbb{S}^n_+)}(x)\cap \mathcal{T}_{\mathcal{A}^{-1}(b)}(x)=(g'(x))^{-1}[\mathcal{T}_{\mathbb{S}^n_+}(g(x))]\cap {\rm Ker}\mathcal{A}$;

 \item[(iii)] for any $v\in\mathcal{N}_{\Gamma}(x)$, $\mathcal{C}_{\Gamma}(x,v)=(g'(x))^{-1}[\mathcal{T}_{\mathbb{S}^n_+}(g(x))]\cap {\rm Ker}\mathcal{A}\cap[\![v]\!]^\perp$;  
	
 \item [(iv)] for any $v\in\mathcal{N}_{\Gamma}(x)$, $\Gamma$ is parabolically regular at $x$ for $v$, $d^2\delta_{\Gamma}(x|v)$ is a proper and lsc function with $ {\rm dom}\,d^2\delta_{\Gamma}(x|v)=\mathcal{C}_{\Gamma}(x,v)$, and for any $w\in\mathcal{C}_{\Gamma}(x,v)$, 
 \begin{equation}\label{Gam-ssubderiv}
 d^2\delta_{\Gamma}(x|v)(w)=\!\!\sup_{(y,S)\in\mathbb{R}^m\times\mathcal{N}_{\mathbb{S}_{+}^n}(g(x))\atop
 \mathcal{A}^*y+\nabla g(x)S=v}\!\!\langle S,D^2g(x)(w,w)- 2(g'(x)w)(g(x))^{\dagger}(g'(x)w)\rangle,
 \end{equation} 
 where the supremum problem on the right hand side has a finite optimal value.

 \item[(v)] If the mapping $\mathcal{G}$ is metrically regular at $(x,(0,0))$, then for any $v\in\mathcal{N}_{\Gamma}(x)$ and $w\in\mathcal{C}_{\Gamma}(x,v)$ the supremum problem in \eqref{Gam-ssubderiv} has a nonempty compact solution set. 
 \end{itemize}  
 \end{proposition}
 \begin{proof}
 {\bf(i)} As the mapping $\mathcal{G}$ is subregular at $(x,(0,0))\!\in\!{\rm gph}\,\mathcal{G}$, from \cite[Section 3.1]{Ioffe08}, 
 \begin{equation}\label{temp-inclusion}
 \widehat{\mathcal{N}}_{g^{-1}(\mathbb{S}^n_+)}(x)+\mathcal{N}_{\mathcal{A}^{-1}(b)}(x)\subset\widehat{\mathcal{N}}_{\Gamma}(x)\subset\mathcal{N}_{\Gamma}(x)\subset\mathcal{N}_{g^{-1}(\mathbb{S}^n_+)}(x)+\mathcal{N}_{\mathcal{A}^{-1}(b)}(x),
 \end{equation}
 where the first inclusion is by the definition of regular normal cones. While using the subregularity of the mapping $\mathcal{H}$ at $(x,0)$ and \cite[Section 3.1]{Ioffe08}, we have
 \[
  \nabla g(x)\widehat{\mathcal{N}}_{\mathbb{S}_{+}^n}(g(x))\subset\widehat{\mathcal{N}}_{g^{-1}(\mathbb{S}^n_+)}(x)\subset\mathcal{N}_{g^{-1}(\mathbb{S}^n_+)}(x)\subset \nabla g(x)\mathcal{N}_{\mathbb{S}_{+}^n}(g(x)). 
 \]
 This, by the convexity of $\mathbb{S}_{+}^n$, means that $\widehat{\mathcal{N}}_{g^{-1}(\mathbb{S}^n_+)}(x)=\mathcal{N}_{g^{-1}(\mathbb{S}^n_+)}(x)=\nabla g(x)\mathcal{N}_{\mathbb{S}_{+}^n}(g(x))$. 
 Together with the above \eqref{temp-inclusion} and $\mathcal{N}_{\!\mathcal{A}^{-1}(b)}(X)={\rm Im}\,\mathcal{A}^{*}$, we get the desired equalities.

 \noindent
 {\bf(ii)-(iii)} The first equality in part (ii) follows the definition of tangent cones and the convexity of $\mathcal{A}^{-1}(b)$. Recall that the mapping $\mathcal{H}$ is subregular at $(x,0)$. Invoking \cite[Proposition 1]{Herion05} leads to $\mathcal{T}_{g^{-1}(\mathbb{S}^n_+)}(x)=(g'(x))^{-1}[\mathcal{T}_{\mathbb{S}^n_+}(g(x))]$. The second equality in part (ii) follows. Part (iii) is immediate by the definition of critical cones and part (ii).

 \medskip
 \noindent
 {\bf(iv)} Fix any $v\!\in\mathcal{N}_{\Gamma}(x)$. The polyhedrality of the set $\mathcal{A}^{-1}(b)$ implies that it is parabolically derivable and parabolically regular. From \cite[Example 3.140]{BS00} and \cite[Theorem 6.2]{Moha2020}, the set $\mathbb{S}_{+}^n$ is parabolically derivable at any $Z\in\mathbb{S}_{+}^n$ for any $W\in\mathcal{T}_{\mathbb{S}_{+}^n}(Z)$, and parabolically regular at any $Z\in\mathbb{S}_{+}^n$ for any $S\in\mathcal{N}_{\mathbb{S}_{+}^n}(Z)$. By \cite[Theorem 5.6]{Moha2020}, $\Gamma$ is parabolically regular at $x$ for $v$, and by \cite[Proposition 5.3 (ii)]{Moha2020} $d^2\delta_{\Gamma}(x|v)$ is a proper lsc function with ${\rm dom}\,d^2\delta_{\Gamma}(x|v)=\mathcal{C}_{\Gamma}(x,v)$. Moreover, by invoking \cite[Theorem 5.6]{Moha2020} with $f(x')=(\mathcal{A}x'\!-b;g(x'))$ for $x'\in\mathbb{X}$ and $\Theta=\{0\}^m\times\mathbb{S}^n_+$, for any $w\in\mathcal{C}_{\Gamma}(x|v)$, it holds  
 \begin{align*}
  d^2\delta_{\Gamma}(x|v)(w)
 &=\sup_{(y,S)\in\mathcal{M}(x,v)}\Big\{\langle S,D^2g(x)(w,w)\rangle+d^2\delta_{\mathbb{S}^n_+}(g(x)|S)(g'(x)w)\Big\}\\  
 &=\sup\limits_{(y,S)\in\mathcal{M}(x,v)}\Big\{\langle S,D^2g(x)(w,w)\rangle-\sigma_{\mathcal{T}_{\mathbb{S}^n_+}^2(g(x),g'(x)w)}(S)\Big\}\\
&=\sup\limits_{(y,S)\in\mathcal{M}(x,v)}\Big\{\langle S,D^2g(x)(w,w)\rangle-2\langle S,(g'(x)w)(g(x))^{\dagger}(g'(x)w)\rangle\Big\}
 \end{align*} 
 where the second equality is due to \cite[Theorem 3.3 (ii)]{Moha2020},  and the third is by Lemma \ref{sigma-Lemma}. Combining the above equality with the expression of $\mathcal{M}(x,v)$ gives equality \eqref{Gam-ssubderiv}. Since $d^2\delta_{\Gamma}(x|v)$ is proper and $w\in{\rm dom}\,d^2\delta_{\Gamma}(x|v)$, the supremum problem on the right hand side of \eqref{Gam-ssubderiv} has a finite optimal value. 
 
 \noindent
 {\bf(v)} Fix any $v\in\mathcal{N}_{\Gamma}(x)$ and $w\in{\rm dom}\,d^2\delta_{\Gamma}(x|v)$. The supremum problem in \eqref{Gam-ssubderiv} is the linear conic program with $\widetilde{C}=2(g'(x)w)(g(x))^{\dagger}(g'(x)w)-\!D^2g(x)(w,w)$: 
 \begin{equation}\label{pconic}
  \sup_{(y,S)\in\mathbb{R}^m\times\mathbb{S}^n}\big\{\langle S,-\widetilde{C}\rangle\ \ {\rm s.t.}\ \ \mathcal{A}^*y+\nabla g(x)S=v,\,S\in\mathcal{N}_{\mathbb{S}_{+}^n}(g(x))\big\}.
 \end{equation}
 After an elementary calculation, the dual of \eqref{pconic} takes the following form
 \begin{equation}\label{dconic}
 \inf_{\zeta\in\mathbb{X}}\big\{\langle \zeta,-v\rangle\ \ {\rm s.t.}\ \  \mathcal{A}\zeta=0,\,g'(x)\zeta-\widetilde{C}\in\mathcal{T}_{\mathbb{S}_{+}^n}(g(x))\big\}.
 \end{equation}
 Let $\widetilde{\mathcal{G}}_{x}$ be the multifunction defined as in \eqref{Gmap} with such $\widetilde{C}$ and $\widetilde{b}=0$.   
 By Remark \ref{remark2-regular} (a), $\widetilde{\mathcal{G}}_{x}$ is metrically regular at any $(z,(0,0))\in{\rm gph}\,\widetilde{\mathcal{G}}_{x}$. Note that the mapping $\widetilde{\mathcal{G}}_{x}$ is convex, i.e., has a convex graph. From \cite[Proposition 2.104]{BS00}, we conclude that the metric regularity of $\widetilde{\mathcal{G}}_{x}$ at any $(z,(0,0))\in{\rm gph}\,\widetilde{\mathcal{G}}_{x}$ is equivalent to \cite[condition (2.312)]{BS00} for $K\!=\{0\}^m\times\mathcal{T}_{\mathbb{S}_{+}^n}(g(x))$ and ${\rm dom}\,f=\mathbb{X}$. Recall that \eqref{pconic} has a finite optimal value. The conclusion follows \cite[Theorem 2.165]{BS00} with (P) taking \eqref{dconic}. 
 \end{proof} 
 
 In the rest of this paper, we assume that the mapping $g\!:\mathbb{X}\to\mathbb{S}^n$ is $\mathbb{S}_{-}^n$-convex, which by \cite[Section 2.3.5]{BS00} is equivalent to the convexity of the multifunction $\mathcal{H}$ defined in \eqref{Hmap} or  the convexity of the function $\mathbb{X}\ni x\mapsto\langle S,g(\cdot)\rangle$ for any $S\in\mathbb{S}_{-}^n$. In this setting, the feasible set $\Gamma$ of problem \eqref{pNSDP} is convex. In addition, for any pair $(X,S)\in{\rm gph}\mathcal{N}_{\mathbb{S}_{+}^n}$, we denote the distinct eigenvalues of $X$ and $S$ by $\mu_1>\mu_2>\cdots>\mu_{p}=0$ and $0=\nu_1>\nu_2>\cdots>\nu_s$, respectively, and introduce the following index sets
 \begin{subequations}
 \begin{equation}\label{alp-gamj}
  \alpha_j\!:=\{i\in[n]\,|\,\lambda_i(X)=\mu_j\}\ {\rm for}\ j\in[p],\ \gamma_j\!:=\{i\in[n]\,|\,\lambda_i(S)=\nu_j\}\ {\rm for}\ j\in[s];
 \end{equation}
 \begin{equation}\label{alp-beta}
 \!\alpha\!:=\!\big\{i\in[n]\,|\,\lambda_i(X)>\!0\big\},\beta\!:=\!\big\{i\in[n]\,|\,\lambda_i(X)\!=0=\!\lambda_i(S)\big\},\gamma\!:=\!\big\{i\in[n]\,|\,\lambda_i(S)\!<\!0\big\}.
 \end{equation}
 \end{subequations}
 Obviously, $\alpha=\bigcup_{j=1}^{p-1}\alpha_j,\gamma=\bigcup_{j=2}^{s}\gamma_j$ and $\beta=\alpha_p\cap\gamma_1=\alpha_p\backslash\gamma=\gamma_1\backslash\alpha$. 
 
 \subsection{Sufficient characterizations of tilt stability}\label{sec3.1}
 
 Now we provide a point-based sufficient characterization for tilt stability of local optimal solutions to problem \eqref{pNSDP}. 
\begin{theorem}\label{Scond1-ptilt}
 Consider a local optimal solution $x^*$ of problem \eqref{pNSDP} with $\nabla^2\varphi(x)\succeq 0$ on an open neighborhood $\mathcal{N}^*\!\subset\mathcal{O}$ of $x^*$. Let $v^*\!:=-\nabla\varphi(x^*)$ and $X^*\!:=\!g(x^*)$. Suppose that $\mathcal{G}$ is metrically regular at $(x^*,(0,0))$. Pick any $(y^*,S^*)\!\in\mathop{\arg\min}_{(y,S)\in\mathcal{M}(x^*,v^*)}{\rm rank}(S)$. Let $\alpha^*,\gamma^*$ and $\beta^*$ be the index sets defined by \eqref{alp-gamj}-\eqref{alp-beta} with $(X,S)=(X^*,S^*)$. Then $x^*$ is a tilt-stable solution of \eqref{pNSDP} whenever ${\rm Ker}\nabla^2\varphi(x^*)\cap{\rm Ker}\mathcal{A}\cap[\![v^*]\!]^{\perp}\cap\Upsilon=\{0\}$ with
 \begin{align*}
 \Upsilon&\!:=\!\bigg\{w\!\in\mathbb{X}\ |\ \exists\, \beta_{-}^*,\beta_{+}^*,\beta_0^*,\widetilde{\gamma}^*\subset\![n],P^*\!\in\mathbb{O}(X^*)\cap\mathbb{O}(S^*)\ {\rm and}\ Q=\!{\rm BlkDiag}(Q_1,\ldots,Q_{p})\\ 
 &\qquad\qquad\qquad  {\rm for}\ Q_j\!\in\mathbb{O}^{|\alpha_j|}\ {\rm s.t.}\  \beta^*\!=\!\beta_{-}^*\cup\beta_{+}^*\cup\beta_0^*\cup[\widetilde{\gamma}^*\backslash\gamma^*]\ {\rm and}\ \widetilde{W}_{[\alpha^*\cup\beta_{+}^*]\widetilde{\gamma}^*}=0,\\ 
 &\qquad\qquad\qquad \widetilde{W}_{\beta_0^*\beta_0^*}\succeq 0,\,\widetilde{W}_{[\beta_{0}^*\cup\beta_{-}^*\cup\widetilde{\gamma}^*][\beta_{-}^*\cup\widetilde{\gamma}^*]}\!=0\ {\rm with}\ \widetilde{W}\!:=(P^*Q)^\top(g'(x^*)w)P^*Q\bigg\}.
 \end{align*}
 \end{theorem}
 \begin{proof}
 The local optimality of $x^*$ to problem \eqref{pNSDP} implies $v^*\in\mathcal{N}_{\Gamma}(x^*)$. According to the metric regularity of $\mathcal{G}$ at $(x^*,(0,0))$ and Remark \ref{remark2-regular} (b), using Proposition \ref{Gamset-prop} (i) shows $\mathcal{M}(x^*,v^*)\ne\emptyset$, so $(y^*,S^*)\in\mathcal{M}(x^*,v^*)$ is well defined with $S^*\in \mathcal{N}_{\mathbb{S}^n_+}(X^*)$ and $\mathcal{A}^*y^*\!+\nabla g(x^*)S^*=v^*$. As $S^*\in \mathcal{N}_{\mathbb{S}^n_+}(X^*)$, the matrices $X^*$ and $S^*$ have a simultaneous ordered spectral decomposition, i.e., there exists $P^*\in\mathbb{O}(X^*)\cap\mathbb{O}(S^*)$ such that  
 \begin{equation}\label{EigDecompXS}
   X^*\!=P^*\begin{pmatrix}
      {\rm Diag}(\lambda^{+}(X^*)) & 0\\
     		0 & 0 \\
   \end{pmatrix}(P^*)^{\top}
  \ {\rm and}\  
  S^*\!=P^*\begin{pmatrix}
     0 & 0\\
     0 & {\rm Diag}(\lambda^{-}(S^*))
 \end{pmatrix}(P^*)^{\top}.
 \end{equation}
 By invoking Lemma \ref{Tiltlemma} with $\vartheta=\varphi$ and $h=\delta_{\Gamma}$, when ${\rm Ker}\nabla^2\varphi(x^*)\cap \mathcal{W}=\{0\}$ with  
 \begin{align}\label{MW-set}
 \mathcal{W}&=\Big\{w\in \mathbb{X}\ |\ {\rm there\ is\ a\ sequence}\ \{((x^k,v^k),w^k)\}_{k\in\mathbb{N}}\subset{\rm gph}\,\mathcal{N}_{\Gamma}\times\mathbb{X}\ {\rm with}\nonumber\\
 &\qquad\qquad\quad \lim_{k\to \infty}(x^k,v^k,w^k)= (x^*,v^*,w)\ {\rm s.t.}\ \lim_{k\to \infty} d^2\delta_{\Gamma}(x^k|v^k)(w^k)=0\Big\},
 \end{align}
 the solution $x^*$ is tilt-stable. Thus, to achieve the conclusion, it suffices to argue that $\mathcal{W}\subset{\rm Ker}\,\mathcal{A}\cap[\![v^*]\!]^{\perp}\cap\Upsilon$. Pick any $w\in\mathcal{W}$. We will prove that $w\in{\rm Ker}\,\mathcal{A}\cap[\![v^*]\!]^{\perp}\cap\Upsilon$. 
 
 {\bf Step 1: to seek the desired index sets $\beta_{-}^*,\beta_{+}^*,\beta_0^*,\widetilde{\gamma}^*\subset\![n]$.}  
 From $w\in\mathcal{W}$ and equation \eqref{MW-set}, there exist sequences $\{(x^k,v^k)\}_{k\in\mathbb{N}}\subset{\rm gph}\,\mathcal{N}_{\Gamma}$ and $\{w^k\}_{k\in\mathbb{N}}\subset\mathbb{X}$ such that
 \begin{equation}\label{XGamWk-equa1}
  \lim_{k\to \infty}(x^k,v^k,w^k)=(x^*,v^*,w)\ \ {\rm and}\ \ 
  \lim_{k\to \infty} d^2 \delta_{\Gamma}(x^k|v^k)(w^k)=0. 
 \end{equation} 
 By the second limit in \eqref{XGamWk-equa1}, there is $\overline{k}\in\mathbb{N}$ such that for all $k\ge\overline{k}$, $w^k\in {\rm dom}\,d^2\delta_{\Gamma}(x^k|v^k)$. Recall that $\{(x^k,v^k)\}_{k\in\mathbb{N}}\subset{\rm gph}\,\mathcal{N}_{\Gamma}$. By Remark \ref{remark1-regular}, the metric regularity of $\mathcal{G}$ at $(x^*,(0,0))$ implies that it is metrically regular at each $(x^k,(0,0))$ for $k\ge\overline{k}$ (if necessary by increasing $\overline{k}$). By Remark \ref{remark2-regular} (b) and Proposition \ref{Gamset-prop} (iv), for each $k\ge\overline{k}$, $d^2\delta_{\Gamma}(x^k|v^k)$ is a proper lsc function with ${\rm dom}\,d^2\delta_{\Gamma}(x^k|v^k)=\mathcal{C}_{\Gamma}(x^k,v^k)$, which together with $w^k\in {\rm dom}\,d^2\delta_{\Gamma}(x^k|v^k)$ and Proposition \ref{Gamset-prop} (iii) leads to the inclusion
 \begin{equation}\label{wk-equa}
  w^k\in (g'(x^k))^{-1}[\mathcal{T}_{\mathbb{S}^n_+}(X^k)]\cap {\rm Ker}\mathcal{A}\cap[\![v^k]\!]^\perp\ {\rm with}\ X^k\!:=g(x^k). 
 \end{equation} 
 Furthermore, by Proposition \ref{Gamset-prop} (v), for each $k\ge\overline{k}$, there exists $(y^k,S^k)\in\mathcal{M}(x^k,v^k)$, i.e. $(y^k,S^k)\in\mathbb{R}^m\times\mathcal{N}_{\mathbb{S}_{+}^n}(X^k)$ with $\mathcal{A}^*y^k+\nabla g(x^k)S^k\!=v^k$, such that 
 \begin{equation}\label{TSubTemp1}
  d^2\delta_{\Gamma}(x^k|v^k)(w^k)=\langle S^k, D^2g(x^k)(w^k,w^k)- 2(g'(x^k)w^k)(X^k)^{\dagger}(g'(x^k)w^k)\rangle.
 \end{equation}
 Since $\mathcal{G}$ is metrically regular at $(x^*,(0,0))$ and $v^*\in\mathcal{N}_{\Gamma}(x^*)$, by Proposition \ref{Mmap-prop}, the mapping $\mathcal{M}$ is locally bounded at $(x^*,v^*)$. Recall that $(x^k,v^k)\to(x^*,v^*)$ as $k\to\infty$. There exists $\widehat{c}_0>0$ such that $\mathcal{M}(x^k,v^k)\subset\big\{(y,S)\in\mathbb{R}^m\times\mathbb{S}^n\,|\ \sqrt{\|y\|^2+\|S\|_F^2}\le \widehat{c}_0\big\}$ for all $k\ge\overline{k}$ (if necessary by increasing $\overline{k}$). From the boundedness of $\{(y^k,S^k)\}_{k\in\mathbb{N}}$ is bounded, there exists an index set $\mathcal{K}\subset\mathbb{N}$ such that $\{(y^k,S^k)\}_{k\in\mathcal{K}}$ is convergent. Denote its limit by $(y,S)$. From the closedness of $\mathcal{M}$, we have $(y,S)\in\mathcal{M}(x^*,v^*)$. Then, the definitions of $(y^*,S^*)$ and the index set $\gamma^*$ imply that 
 $\gamma^*\subset\widetilde{\gamma}^*:=\{i\in[n]\,|\,\lambda_i(S)<0\}$.  
 
 For each $k\ge\overline{k}$, by recalling that $S^k\in\mathcal{N}_{\mathbb{S}_{+}^n}(X^k)$, the matrices $X^k$ and $S^k$ have a simultaneous ordered eigenvalue decomposition, i.e., there exists $P^k\in\mathbb{O}^n$ such that 
 \begin{align}\label{EigDecompXk}
  X^k=P^k\begin{pmatrix}
  {\rm Diag}(\lambda^{+}(X^k)) & 0\\
    0 & 0\\
  \end{pmatrix}(P^k)^\top\ \ {\rm and}\ \ 
S^k=P^k\begin{pmatrix}
    0 & 0\\
   0 & {\rm Diag}(\lambda^{-}(S^k))
  \end{pmatrix}(P^k)^\top,
\end{align}
and we write $\alpha_k\!:=\!\{i\in [n]\,|\, \lambda_i(X^k)>0\},\,\beta_k\!:=\!\{i\in [n]\,|\, \lambda_i(X^k)\!=\lambda_i(S^k)=0\}$ and 
$\gamma_k\!:=\{i\in [n]\,|\, \lambda_i(S^k)\!<0\}$. By considering that $\lambda(X^k)\to \lambda(X^*)$ and $\lambda(S^k)\to\lambda(S)$ as $\mathcal{K}\ni k\to\infty$, the following two index sets are well defined:
\begin{subequations}
\begin{equation}\label{beta+}
\!\beta_{+}^*:=\big\{j\in [n]\,|\, \lim_{\mathcal{K}\ni k\to \infty}\lambda_j(X^k)=0\ {\rm with}\ \lambda_j(X^k)>0\ {\rm for}\ k\in\mathcal{K}\ {\rm large\ enough}\big\},
 \end{equation}
 \begin{equation}\label{beta-}
 \beta_{-}^*:=\{j\in [n]\,|\, \lim_{\mathcal{K}\ni  k\to \infty}\lambda_j(S^k)=0\ {\rm with}\ \lambda_j(S^k)<0\ {\rm for}\ k\in\mathcal{K}\ {\rm large\ enough}\}. 
 \end{equation}
\end{subequations}
Notice that $S\in\mathcal{N}_{\mathbb{S}_{+}^n}(X^*)$ with ${\rm rank}(S)\ge{\rm rank}(S^*)$. By the definition of $\beta^*$, we have 
\begin{equation*}
 \beta_{+}^*\cup\beta_{-}^*\subset\beta_{S}^*:=\{i\in[n]\,|\,\lambda_i(X^*)=0=\lambda_i(S)\}\subset\beta^*.
\end{equation*}
 Let 
$\beta_0^*\!:=\beta_{S}^*\backslash(\beta_{+}^*\cup\beta_{-}^*)$. Then,  $\beta^*=\beta_{S}^*\cup[\widetilde{\gamma}^*\backslash\gamma^*]=\beta_{+}^*\cup\beta_{-}^*\cup\beta_0^*\cup[\widetilde{\gamma}^*\backslash\gamma^*]$.
Moreover, it is easy to deduce that for each $\mathcal{K}\ni k\ge\overline{k}$ (if necessary by increasing $\overline{k}$), 
\begin{equation}\label{indexk} \alpha_k\!=\alpha^*\cup\beta_{+}^*,\,\beta_k=\beta_{S}^*\backslash(\beta_{+}^*\cup\beta_{-}^*)=\beta_0^*\ \ {\rm and}\ \ \gamma_k=\beta_{-}^*\cup\widetilde{\gamma}^*.
\end{equation} 

{\bf Step 2: to seek the desired block orthogonal matrix $Q$.} From \eqref{wk-equa}, we have $\mathcal{A}w^k=0$ and $\langle v^k,w^k\rangle=0$ for each $k\ge\overline{k}$, which by $\lim\limits_{k\to\infty}w^k\!=\!w$  and $\lim\limits_{k\to\infty}v^k\!=\!v^*$ implies
\begin{equation}\label{w-relation}
 \mathcal{A} w=0\ \ {\rm and}\ \ \langle v^*,w\rangle=0. 
\end{equation}
For each $k\!\ge\!\overline{k}$, recalling that $\langle v^k,w^k\rangle=0$, $(y^k,S^k)\in\mathcal{M}(x^k,v^k)$ and $\mathcal{A}w^k=0$, we have
\[
  0=\langle w^k,v^k\rangle=\langle w^k,\mathcal{A}^*y^k\!+\!\nabla g(x^k)S^k\rangle=\langle g'(x^k)w^k, S^k\rangle.
\] 
Along with $g'(x^k)w^k\in\!\mathcal{T}_{\mathbb{S}_{+}^n}(X^k)$ by \eqref{wk-equa}, we get $g'(x^k)w^k\!\in\mathcal{C}_{\mathbb{S}_{+}^n}(X^k,S^k)$ for each $k\ge\overline{k}$. Now, for each $\mathcal{K}\ni k\ge\overline{k}$, by \eqref{EigDecompXk} and \eqref{indexk}, invoking \eqref{SDCritial} with $(X,S)=(X^k,S^k)$ leads to 
\begin{align}\label{WkStruc}
 \widetilde{W}^k:=(P^k)^\top (g'(x^k)w^k) P^k
 =\begin{pmatrix}
\widetilde{W}^k_{11} & \widetilde{W}^k_{12}& \widetilde{W}^k_{13} & \widetilde{W}^k_{14}\\
 (\widetilde{W}^k_{12})^\top & \widetilde{W}^k_{22} & 0 & 0\\
 (\widetilde{W}^k_{13})^{\top} &0 & 0 & 0\\
 (\widetilde{W}^k_{14})^{\top} &0 & 0 & 0
 \end{pmatrix},
\end{align}
where $\widetilde{W}^k_{11}\in \mathbb{S}^{|\alpha^*\cup\beta_{+}^*|},\widetilde{W}^k_{12}\in \mathbb{R}^{|\alpha^*\cup\beta_{+}^*|\times |\beta_0^*|},\widetilde{W}^k_{13}\in \mathbb{R}^{|\alpha^*\cup\beta_{+}^*|\times |\beta^*_-|},\widetilde{W}^k_{14}\in \mathbb{R}^{|\alpha^*\cup\beta_{+}^*|\times |\widetilde{\gamma}^*|}$ and $\widetilde{W}^k_{22}\in\mathbb{S}_{+}^{|\beta_0^*|}$.
Combining \eqref{WkStruc} with the above \eqref{TSubTemp1}-\eqref{EigDecompXk} yields that for each $\mathcal{K}\ni k\ge\overline{k}$,
\begin{align*}
 d^2 \delta_{\Gamma}(x^k|v^k)(w^k)
 &=\langle S^k, D^2g(x^k)(w^k,w^k)\rangle-2\langle \Lambda^{-}(S^k), (\widetilde{W}^k_{13})^\top [\Lambda^{+}(X^k)]^{\dagger}\widetilde{W}^k_{13}\rangle\\
 &=\langle S^k, D^2g(x^k)(w^k,w^k)\rangle-2\sum_{j\in \widetilde{\gamma}^*\cup\beta_{-}^*}\sum_{i\in\alpha^*\cup \beta_{+}^*}\frac{\lambda_{j}(S^k)}{\lambda_i(X^k)}(\widetilde{W}_{ji}^k)^2, 
\end{align*}
where the second equality is obtained by using $\alpha_k=\alpha^*\cup \beta_{+}^*$ and $\gamma_k=\widetilde{\gamma}^*\cup\beta_{-}^*$. Recall that $\mathbb{S}_{-}^n$-convexity of $g$ implies the convexity of the function $\mathbb{X}\ni x\mapsto\langle S,g(\cdot)\rangle$. Passing the limit $\mathcal{K}\ni k\to\infty$
to the above equality and using the second limit in \eqref{XGamWk-equa1} leads to 
\[
  0\ge \lim_{\mathcal{K}\ni k\to\infty} \sum_{j\in \widetilde{\gamma}^*\cup\beta_{-}^*}\sum_{i\in\alpha^*\cup \beta_{+}^*}\frac{\lambda_{j}(S^k)}{\lambda_i(X^k)}(\widetilde{W}_{ji}^k)^2=\frac{1}{2}\langle S, D^2g(x^*)(w,w)\rangle\ge 0,
\]
where the first inequality is due to $\lambda_j(S^k)<0$ for $j\in\!\widetilde{\gamma}^*\cup\beta_{-}^*$ and $\lambda_i(X^k)>0$ for $i\in\alpha^*\cup\beta_{+}^*$. The above equation implies that   $\lim_{\mathcal{K}\ni k\to\infty}\frac{\lambda_{j}(S^k)}{\lambda_i(X^k)}(\widetilde{W}_{\!ji}^k)^2=0$ for each $(i,j)\in\!(\alpha^*\cup \beta_{+}^*)\times(\widetilde{\gamma}^*\cup\beta_{-}^*)$. By the definitions of $\alpha^*,\beta_{+}^*$ and $\widetilde{\gamma}^*$, we deduce that    
$\lim_{\mathcal{K}\ni k\to \infty}\widetilde{W}_{\!ji}^k=0$ for each $(i,j)\in(\alpha^*\cup \beta_{+}^*)\times\widetilde{\gamma}^*$. Note that $\{P^k\}_{k\in\mathcal{K}}$ is bounded. There exists an infinite index set $\mathcal{K}_1\subset\mathcal{K}$ such that $\{P^k\}_{k\in\mathcal{K}_1}$ is convergent with limit denoted by $\overline{P}$. Moreover, from the eigenvalue decomposition of $X^k$ in \eqref{EigDecompXk} and $\lim_{k\to\infty}X^k=X^*$, we have $\overline{P}\in\mathbb{O}(X^*)$. Recall that $\lim_{k\to\infty}x^k=x^*,\,\lim_{k\to\infty}w^k=w$ and $g$ is continuously differentiable. Passing the limit $\mathcal{K}_1\ni k\to\infty$ to equality \eqref{WkStruc} results in 
\begin{equation}\label{Wt-equa}
 \widetilde{W}:=\overline{P}^{\top}(g'(x^*)w)\overline{P}
  =\begin{pmatrix}
    \widetilde{W}_{11} & \widetilde{W}_{12} & \widetilde{W}_{13} & 0_{[\alpha^*\cup\beta_{+}^*]\widetilde{\gamma}^*}\\
    (\widetilde{W}_{12})^{\top}&\widetilde{W}_{22}& 0 &0\\
    (\widetilde{W}_{13})^{\top}&0&0&0\\
    0_{\widetilde{\gamma}^*[\alpha^*\cup\beta_{+}^*]} & 0 &0 &0
  \end{pmatrix}
\end{equation}
with $\widetilde{W}_{11}\!\in \mathbb{S}^{|\alpha^*\cup\beta_{+}^*|},\widetilde{W}_{12}\in \mathbb{R}^{|\alpha^*\cup\beta_{+}^*|\times |\beta_0^*|},\widetilde{W}^k_{13}\in \mathbb{R}^{|\alpha^*\cup\beta_{+}^*|\times |\beta^*_-|}$ and $\widetilde{W}_{22}\in\mathbb{S}_{+}^{|\beta_0^*|}$. 
From
\[
  X^*=P^*\begin{pmatrix}
     {\rm Diag}(\lambda^{+}(X^*)) & 0\\
     		0 & 0 \\
   \end{pmatrix}(P^*)^{\top}=\overline{P}\begin{pmatrix}
      {\rm Diag}(\lambda^{+}(X^*)) & 0\\
     		0 & 0 \\
   \end{pmatrix}\overline{P}^{\top},
\]
there exists a block diagonal $Q={\rm BlkDiag}(Q_1,\ldots,Q_{p})$ with $Q_j\in\mathbb{O}^{|\alpha_j|}$ for $j\in[p]$ such that $\overline{P}=P^*Q$. Together with \eqref{Wt-equa}, the matrix $Q$ meets the requirement of $\Upsilon$.    
The above arguments show that $w\in{\rm Ker}\,\mathcal{A}\cap[\![v^*]\!]^{\perp}\cap\Upsilon$. Consequently, the desired conclusion follows. 
\end{proof}
\begin{corollary}\label{Scorollary1-ptilt}
 For the set $\Upsilon$ appearing in Theorem \ref{Scond1-ptilt}, it holds  $\Upsilon=\Upsilon^*$, where $\Upsilon^*$ is defined as
 \begin{align}\label{Upsion-star}
 \Upsilon^*&:=\!\bigg\{w\!\in\mathbb{X}\ |\ \exists\, P^*\!\in\mathbb{O}(X^*)\cap\mathbb{O}(S^*)\ {\rm and}\ Q=\!{\rm BlkDiag}(Q_1,\ldots,Q_{p})\ {\rm for}\ Q_j\!\in\mathbb{O}^{|\alpha_j|}\nonumber\\ 
 &\qquad\qquad\quad{\rm s.t.}\ R:=P^*Q\in\mathbb{O}(X^*)\ {\rm and}\ [R_{\alpha^*}\  R_{\beta^*}\ R_{\gamma^*}]^{\top}(g'(x^*)w)R_{\gamma^*}=0\bigg\},
 \end{align}
\end{corollary}
\begin{proof}
 We first prove $\Upsilon\subset\Upsilon^*$. Pick any $w\in\Upsilon$. There exist index sets $\beta_{+}^*,\beta_{0}^*,\beta_{-}^*,\widetilde{\gamma}^*\subset[n]$, $P^*\in\mathbb{O}(X^*)\cap\mathbb{O}(S^*)$ and $Q={\rm BlkDiag}(Q_1,\ldots,Q_{p})$ for $Q_j\in\mathbb{O}^{|\alpha_j|}$ such that 
 $\beta^*\!=\beta_{-}^*\cup\beta_{+}^*\cup\beta_0^*\cup[\widetilde{\gamma}^*\backslash\gamma^*]$ and $\widetilde{W}_{[\alpha^*\cup\beta_{+}^*]\widetilde{\gamma}^*}\!=0,\widetilde{W}_{\beta_0^*\beta_0^*}\succeq 0\ {\rm and}\ \widetilde{W}_{[\beta_{0}^*\cup\beta_{-}^*\cup\widetilde{\gamma}^*][\beta_{-}^*\cup\widetilde{\gamma}^*]}\!=0$ for $\widetilde{W}\!=(P^*Q)^\top(g'(x^*)w)P^*Q$. Let $R=P^*Q$. By the definition of $Q$, it is easy to check that $R\in\mathbb{O}(X^*)$. Furthermore,  
 \[
  R_{\alpha^*}^{\top}(g'(x^*)w)R_{\gamma^*}=0,\,R_{\beta^*}^{\top}(g'(x^*)w)R_{\gamma^*}=0\ {\rm and}\ R_{\gamma^*}^{\top}(g'(x^*)w)R_{\gamma^*}=0.
  \]
 Thus, $w\in\Upsilon^*$. By the arbitrariness of $w\in\Upsilon$, the inclusion $\Upsilon\subset\Upsilon^*$ follows. 

 For the converse inclusion, pick any $w\in\Upsilon^*$. Set $\beta_{0}^*=\beta_{-}^*=\emptyset$, $\beta_{+}^*=\beta^*$ and $\widetilde{\gamma}^*=\gamma^*$. Clearly,  $\beta^*=\beta_{+}^*\cup\beta_{-}^*\cup\beta_{0}^*\cup[\widetilde{\gamma}^*\backslash\gamma^*]$. Then, from the expression of $\Upsilon^*$, we conclude that $w\in\Upsilon$. By the arbitrariness of $w\in\Upsilon^*$, the inclusion $\Upsilon^*\subset\Upsilon$ holds.
 \end{proof}

 Theorem \ref{Scond1-ptilt} provide a point-based sufficient characterization for the tilt-stable local minimizers of \eqref{pNSDP} by virtue of a multiplier pair of minimum rank. Its significance is theoretical because seeking such a multiplier pair is equivalent to finding a global optimal solution of the NP-hard problem 
 \begin{equation*}
 \min_{y\in\mathbb{R}^m,S\in\mathbb{S}^n}\big\{{\rm rank}(S)\ \ {\rm s.t.}\ \mathcal{A}^*y+\nabla\!g(x^*)S=-\nabla\varphi(x^*),\,S\in\mathcal{N}_{\mathbb{S}_{+}^n}(g(x^*))\big\}. 
 \end{equation*}
 Next, under a suitable restriction on the multiplier set, we provide another point-based sufficient characterization for the tilt-stable local optimal solution of problem \eqref{pNSDP}. 
\begin{theorem}\label{Scond2-ptilt}
 Consider a local optimal solution $x^*$ of problem \eqref{pNSDP} with $\nabla^2\varphi(x)\succeq 0$ on an open neighborhood $\mathcal{N}^*\!\subset\mathcal{O}$ of $x^*$. Let $v^*\!:=\!-\nabla\varphi(x^*)$ and $X^*\!:=g(x^*)$. Suppose that the multifunction $\mathcal{G}$ is metrically regular at $(x^*,(0,0))$, and that $\mathcal{M}(x^*,v^*)\!=(\mathcal{A}^*)^{-1}(v^*\!-\nabla\!g(x^*)S^*)\times\{S^*\}$ for some $S^*\in\mathcal{N}_{\mathbb{S}_{+}^n}(X^*)$. Let $\alpha^*,\gamma^*$ and $\beta^*$ be the index sets defined by \eqref{alp-gamj}-\eqref{alp-beta} with $(X,S)=(X^*,S^*)$. Then, the solution $x^*$ is tilt-stable whenever ${\rm Ker}\,\nabla^2\varphi(x^*)\cap{\rm Ker}\,\mathcal{A}\cap\widehat{\Upsilon}=\{0\}$ with 
 \begin{align*}
 \widehat{\Upsilon}&\!:=\bigg\{w\!\in\mathbb{X}\ |\ \exists\, \beta_{-}^*,\beta_{+}^*,\beta_0^*,\widetilde{\gamma}^*\subset\![n],P^*\!\in\mathbb{O}(X^*)\cap\mathbb{O}(S^*)\ {\rm and}\ Q=\!{\rm BlkDiag}(Q_1,\ldots,Q_{p})\\ 
 &\qquad {\rm with}\ Q_j\in\mathbb{O}^{|\alpha_j|}\ {\rm and}\ Q_p=\!{\rm BlkDiag}(Q_{p}^{1},\ldots,Q_{p}^{s})\ {\rm with}\  Q_{p}^{1}=\mathbb{O}^{|\beta^*|}, Q_p^{j}\in\mathbb{O}^{|\gamma_j|}\ {\rm for}\\  
 &\qquad j\in\{2,\ldots,s\}\ {\rm s.t.}\  \beta^*\!=\!\beta_{-}^*\cup\beta_{+}^*\cup\beta_0^*\cup[\widetilde{\gamma}^*\backslash\gamma^*],
 P^*Q\in\mathbb{O}(X^*)\cap\mathbb{O}(S^*),\widetilde{W}_{\beta_0^*\beta_0^*}\succeq 0,\\ 
 &\qquad\quad\widetilde{W}_{[\beta_{0}^*\cup\beta_{-}^*\cup\widetilde{\gamma}^*][\beta_{-}^*\cup\widetilde{\gamma}^*]}\!=0\ {\rm and}\ \widetilde{W}_{[\alpha^*\cup\beta_{+}^*]\widetilde{\gamma}^*}=0\ {\rm for}\ \widetilde{W}\!:=(P^*Q)^\top(g'(x^*)w)P^*Q\bigg\}.
 \end{align*}
 \end{theorem}
 \begin{proof}
 Since the mapping $\mathcal{G}$ is metrically regular at $(x^*,(0,0))$, we have $\mathcal{M}(x^*,v^*)\ne\emptyset$, so there exists $y^*\in\mathbb{R}^m$ such that $\mathcal{A}^*y^*+\nabla\!g(x^*)S^*=v^*$. From $S^*\in\mathcal{N}_{\mathbb{S}_{+}^n}(X^*)$, there exists $P^*\in\mathbb{O}(X^*)\cap\mathbb{O}(S^*)$ such that $X^*$ and $S^*$ satisfy \eqref{EigDecompXS}. In addition, from Step 1 in the proof of Theorem \ref{Scond1-ptilt}, $(y,S)\in\mathcal{M}(x^*,v^*)=(\mathcal{A}^*)^{-1}(v^*\!-\!\nabla\!g(x^*)S^*)\times\{S^*\}$. This implies $S=S^*$. Together with \eqref{EigDecompXS} and the proof in Step 2 of Theorem \ref{Scond1-ptilt}, it holds  
  \[
  S^*=P^*\begin{pmatrix}
      0 & 0\\
     0 & {\rm Diag}(\lambda^{-}(S^*)) \\
   \end{pmatrix}(P^*)^{\top}=\overline{P}\begin{pmatrix}
      0& 0\\
     0 & {\rm Diag}(\lambda^{-}(S^*))\\
   \end{pmatrix}\overline{P}^{\top}.
  \]
Recall that $Q=(P^*)^{\top}\overline{P}={\rm BlkDiag}(Q_1,\ldots,Q_{p})$ with $Q_j\in\mathbb{O}^{|\alpha_j|}$ and $\beta^*=\alpha_p\backslash\gamma^*$. Write
\[
  Q_p=\begin{pmatrix}
       Q_{p}^{11}&Q_{p}^{12}\\
       Q_{p}^{21}&Q_{p}^{22}
      \end{pmatrix}\ \ {\rm with}\ \ Q_{p}^{11}\in\mathbb{R}^{|\beta^*|\times|\beta^*|}\ {\rm and}\ Q_{p}^{22}\in\mathbb{R}^{|\gamma^*|\times|\gamma^*|}.
\]
From the above two equations, it is not difficult to obtain that 
\[
  \begin{pmatrix}
       Q_{p}^{11}&Q_{p}^{12}\\
       Q_{p}^{21}&Q_{p}^{22}
 \end{pmatrix}\begin{pmatrix}
       0_{\beta^*\beta^*}& 0\\
       0&{\rm Diag}(\lambda^{-}(S^*))
      \end{pmatrix}=\begin{pmatrix}
        0_{\beta^*\beta^*}& 0\\
       0&{\rm Diag}(\lambda^{-}(S^*))
      \end{pmatrix}\begin{pmatrix}
       Q_{p}^{11}&Q_{p}^{12}\\
       Q_{p}^{21}&Q_{p}^{22}
 \end{pmatrix}.
\]
This implies that $Q_{p}^{12}=Q_{p}^{21}=0$ and $Q_{p}^{22}={\rm BlkDiag}(Q_{p}^2,\ldots,Q_{p}^{s})$ with $Q_{p}^j\in\mathbb{R}^{|\gamma_j|\times|\gamma_j|}$ for $j\in\{2,\ldots,s\}$. Take $Q_{p}^{1}=Q_{p}^{11}$. From the orthogonality of $Q_p$, we have $Q_{p}^j\in\mathbb{O}^{|\gamma_j|}$ for $j\in\{2,\ldots,s\}$ and $Q_{p}^{1}\in\mathbb{O}^{|\beta^*|}$. Then, following the proof of Theorem \ref{Scond1-ptilt}, we conclude that $x^*$ is a tilt-stable solution of \eqref{pNSDP} whenever ${\rm Ker}\nabla^2\varphi(x^*)\cap{\rm Ker}\mathcal{A}\cap[\![v^*]\!]^{\perp}\cap\widehat{\Upsilon}=\{0\}$. 
The rest only needs to argue that ${\rm Ker}\,\mathcal{A}\cap\widehat{\Upsilon}\subset[\![v^*]\!]^{\perp}$. Indeed, by the local optimality of $x^*$ and Proposition \ref{Gamset-prop} (i), $v^*\in\nabla g(x^*)\mathcal{N}_{\mathbb{S}_{+}^n}(X^*)+{\rm Im}\mathcal{A}^*$, so there exist $S\in\mathcal{N}_{\mathbb{S}_{+}^n}(X^*)$ and $\xi\in{\rm Im}\mathcal{A}^*$ such that $v^*=\nabla g(x^*)S+\xi$. Pick any $w\in{\rm Ker}\,\mathcal{A}\cap\widehat{\Upsilon}$. Then, 
\begin{equation}\label{temp-vw}
 \langle w,v^*\rangle=\langle w,\nabla g(x^*)S+\xi\rangle=\langle g'(x^*)w,S\rangle.
\end{equation}
 According to $w\in\widehat{\Upsilon}$, there exist $\beta_{-}^*,\beta_{+}^*,\beta_0^*,\widetilde{\gamma}^*\subset\![n],P^*\!\in\mathbb{O}(X^*)\cap\mathbb{O}(S^*)$ and $Q$ as in $\widehat{\Upsilon}$ such that 
 \[
 \beta^*\!=\!\beta_{-}^*\cup\beta_{+}^*\cup\beta_0^*\cup[\widetilde{\gamma}^*\backslash\gamma^*],\widetilde{W}_{[\alpha^*\cup\beta_{+}^*]\widetilde{\gamma}^*}=0, \widetilde{W}_{\beta_0^*\beta_0^*}\succeq 0\ \ {\rm and}\ \ \widetilde{W}_{[\beta_{0}^*\cup\beta_{-}^*\cup\widetilde{\gamma}^*][\beta_{-}^*\cup\widetilde{\gamma}^*]}\!=0
 \]
 with $\widetilde{W}=R^\top(g'(x^*)w)R$ for $R=P^*Q\in\mathbb{O}(X^*)\cap\mathbb{O}(S^*)$. From $S\in\mathcal{N}_{\mathbb{S}_{+}^n}(X^*)$ and equation \eqref{SDCNor}, it is easy to check that $\langle g'(x^*)w,S\rangle=\langle\widetilde{W},R^{\top}SR\rangle=0$. Then, $w\in[\![v^*]\!]^{\perp}$ follows from \eqref{temp-vw}. The inclusion ${\rm Ker}\mathcal{A}\cap\widehat{\Upsilon}\subset[\![v^*]\!]^{\perp}$ holds by the arbitrariness of $w$. The proof is completed.
 \end{proof}

 By following the proof of Corollary \ref{Scorollary1-ptilt}, we can obtain the following result for the set $\widehat{\Upsilon}$.   
\begin{corollary}\label{Scorollary2-ptilt}
 For the set $\widehat{\Upsilon}$ appearing in Theorem \ref{Scond1-ptilt}, the following relations hold 
 \begin{align*}
 \widehat{\Upsilon}&=\!\bigg\{w\!\in\mathbb{X}\,|\, \exists\,P^*\!\in\mathbb{O}(X^*)\cap\mathbb{O}(S^*)\ {\rm and}\ Q\!=\!{\rm BlkDiag}(Q_1,\ldots,Q_{p})\ {\rm for}\ Q_j\in\mathbb{O}^{|\alpha_j|}\\ 
 &\qquad\quad {\rm and}\ Q_p=\!{\rm BlkDiag}(Q_{p}^{1},\ldots,Q_{p}^{s})\ {\rm with}\  Q_{p}^{1}=\mathbb{O}^{|\beta^*|}, Q_p^{j}\in\mathbb{O}^{|\gamma_j|}\ {\rm for}\ j=2,\ldots,s\\  
 &\qquad\quad {\rm s.t.}\  
 R:=P^*Q\in\mathbb{O}(X^*)\cap\mathbb{O}(S^*)\ {\rm and}\ [R_{\alpha^*}\  R_{\beta^*}\ R_{\gamma^*}]^{\top}(g'(x^*)w)R_{\gamma^*}=0\bigg\}\\
 &=\!\bigg\{w\!\in\mathbb{X}\ |\ \exists\,R\in\mathbb{O}(X^*)\cap\mathbb{O}(S^*)\ \ {\rm s.t.}\ [R_{\alpha^*}\  R_{\beta^*}\ R_{\gamma^*}]^{\top}(g'(x^*)w)R_{\gamma^*}=0\bigg\}:=\widehat{\Upsilon}^{*}. 
 \end{align*}
 \end{corollary}
\begin{remark}
 Theorems \ref{Scond1-ptilt} and \ref{Scond2-ptilt} provide point-based sufficient characterizations for the tilt-stable local optimal solutions of \eqref{pNSDP} by leveraging a multiplier pair of minimum rank and imposing a suitable restriction on the multiplier set, respectively. The restriction on the multiplier set $\mathcal{M}(x^*,v^*)$ in Theorem \ref{Scond2-ptilt} is weaker than the uniqueness assumption on $\mathcal{M}(x^*,v^*)$. Note that $\Phi$ is prox-regular and subdifferentially continuous on $\Gamma$. By \cite[Theorem 3.3]{Drusvyatskiy13}, they also provide the sufficient characterizations for the strong metric regularity of the subdifferential mapping $\partial\Phi=\nabla\varphi+\mathcal{N}_{\Gamma}$ at the point in question. To our knowledge, these are the first point-based sufficient characterizations for the tilt stability of nonconvex SDPs without constraint nondegeneracy. 
\end{remark}
 \subsection{Necessary characterization of tilt stability}\label{sec3.2}

 In this section, for a class of affine mapping $g(x)\!:=\mathcal{B}x\!-\!B$, where $\mathcal{B}\!:\mathbb{X}\to\mathbb{S}^n$ is a linear operator and $B\in\mathbb{S}^n$ is a given matrix, we provide the necessary characterization for the tilt stability of local optimal solutions to problem \eqref{pNSDP}. 
\begin{theorem}\label{Ncond1-ptilt}
 Consider a local optimal solution $x^*$ of \eqref{pNSDP}. Let 
 $v^*\!:=-\nabla\varphi(x^*)$ and $X^*\!:=g(x^*)$. Suppose that  $\mathcal{G}$ is metrically regular at $(x^*,(0,0))$. Pick any $(y^*,S^*)\!\in\mathcal{M}(x^*,v^*)$, and let $\alpha^*,\beta^*$ and $\gamma^*$ be defined by \eqref{alp-gamj}-\eqref{alp-beta} with $(X,S)=\!(X^*,S^*)$. If $\mathcal{B}\mathbb{X}\supset\mathcal{T}_{\mathbb{S}_{+}^n}(X^*)\cap{\rm Sp}(\mathcal{N}_{\mathbb{S}_{+}^n}(-S^*))\!:=\!K^*$, the mapping $\mathcal{B}$ is injective on the set $\mathcal{B}^{-1}(K^*)$, and either of the following two conditions holds:
 \[
  {\bf(A)}\ \ {\rm Sp}(\mathcal{N}_{\mathbb{S}_{+}^n}(X^*))\cap(\mathcal{B}^*)^{-1}({\rm Im}\mathcal{A}^*)=\{0\}\ \ {\rm and}\ \ {\bf(B)}\quad\mathcal{B}^*[{\rm Sp}(\mathcal{N}_{\mathbb{S}_{+}^n}(X^*))]\subset\mathcal{A}^*\mathbb{R}^m,
 \]
 then the tilt stability of solution $x^*$ implies that ${\rm Ker}\,\nabla^2\varphi(x^*)\cap{\rm Ker}\,\mathcal{A}\cap\widetilde{\Upsilon}=\{0\}$ with
 \begin{align*}
 \widetilde{\Upsilon}&\!:=\!\bigg\{w\!\in\mathbb{X}\ |\ \exists\, \beta_{-}^*,\beta_{+}^*,\beta_0^*,\widetilde{\gamma}^*\subset[n]\ {\rm with }\ \gamma^*\subset \widetilde{\gamma}^*\ {\rm and}\ P^*\in\mathbb{O}(X^*)\cap\mathbb{O}(S^*)\ {\rm such\ that}\\ &\qquad\qquad\qquad\beta^*=\beta_{-}^*\cup\beta_{+}^*\cup\beta_0^*\cup[\widetilde{\gamma}^*\backslash\gamma^*],\, \widehat{W}_{[\alpha^*\cup\beta_{+}^*]\widetilde{\gamma}^*}=0,\,\widehat{W}_{\beta_0^*\beta_0^*}\succeq 0\\ 
 &\qquad\qquad\qquad {\rm and}\ \widehat{W}_{[\beta_{0}^*\cup\beta_{-}^*\cup\widetilde{\gamma}^*][\beta_{-}^*\cup\widetilde{\gamma}^*]}\!=0 \ {\rm with}\ 
  \widehat{W}\!:=(P^*)^\top(\mathcal{B}w)P^*\bigg\}\\
 &=\!\bigg\{w\!\in\mathbb{X}\ |\ \exists\,  P^*\in\mathbb{O}(X^*)\cap\mathbb{O}(S^*)\ {\rm s.t.}\ [P_{\alpha^*}^*\  P_{\beta^*}^*\ P_{\gamma^*}^*]^{\top}(\mathcal{B}w)P_{\gamma^*}^*=0\bigg\}:=\widetilde{\Upsilon}^*.
 \end{align*}  
 \end{theorem}
 \begin{proof}
 According to Lemma \ref{Tiltlemma}, it suffices to prove that ${\rm Ker}\mathcal{A}\cap\widetilde{\Upsilon}\subset\mathcal{W}$, where $\mathcal{W}$ is the set defined in \eqref{MW-set}. Pick any $w\in{\rm Ker}\mathcal{A}\cap\widetilde{\Upsilon}$. Then, $\mathcal{A}w=0$, and by the definition of $\widetilde{\Upsilon}$, there exist index sets $\beta_{+}^*,\beta_{0}^*,\beta_{-}^*,\widetilde{\gamma}^*\subset[n]$ and $P^*\in\mathbb{O}(X^*)\cap\mathbb{O}(S^*)$ such that $\beta^*\!=\!\beta_{-}^*\cup\beta_0^*\cup\beta_{+}^*\cup\gamma_0^*$ with $\gamma_0^*:=\widetilde{\gamma}^*\backslash\gamma^*$ and 
\begin{equation}\label{Wtilde-equa-d}
 \widehat{W}\!:=(P^*)^\top (\mathcal{B}w)P^*=\begin{pmatrix}
  \widehat{W}_{11}& \widehat{W}_{12} & \widehat{W}_{13} & 0_{[\alpha^*\cup\beta_{+}^*]\widetilde{\gamma}^*}\\
  \widehat{W}_{12}^{\top}&\widehat{W}_{22} &0  &  0\\
  \widehat{W}_{13}^{\top}& 0 & 0 & 0\\
  0_{\widetilde{\gamma}^*[\alpha^*\cup\beta_{+}^*]} & 0 & 0 & 0
  \end{pmatrix}
\end{equation}
 with $\widehat{W}_{11}\in\mathbb{S}^{|\alpha^*\cup\beta_{+}^*|},\widehat{W}_{12}\in \mathbb{R}^{|\alpha^*\cup\beta_{+}^*|\times |\beta_{0}^*|},\widehat{W}_{13}\in \mathbb{R}^{|\alpha^*\cup\beta_{+}^*|\times |\beta_{-}^*|}$ and $\widehat{W}_{22}\in\mathbb{S}^{|\beta_{0}^*|}_+$. From $P^*\in\mathbb{O}(X^*)\cap\mathbb{O}(S^*)$, the matrices $X^*$ and $S^*$ have the following eigenvalue decomposition
 \begin{align}\label{EigDecomp}
  X^*\!=P^*\begin{pmatrix}
   \Lambda^{+}(X^*) & 0\\
   0 & 0\\
  \end{pmatrix}(P^*)^{\top}\ {\rm and}\  
  S^*\!=P^*\begin{pmatrix}
      0 & 0 \\
     0 & \Lambda^{-}(S^*)
 \end{pmatrix}(P^*)^{\top}
 \end{align}
 with $\Lambda^{+}(X^*)={\rm Diag}(\lambda^{+}(X^*))$ and $\Lambda^{-}(S^*)={\rm Diag}(\lambda^{-}(S^*))$. As $(y^*,S^*)\in\mathcal{M}(x^*,v^*)$, it holds $S^*\in\mathcal{N}_{\mathbb{S}_{+}^n}(X^*)$. We proceed the proof by the following three steps.
 
 \noindent
 {\bf Step 1: to construct a sequence $\{((x^k,v^k),w^k)\}_{k\in\mathbb{N}}\subset{\rm gph}\,\mathcal{N}_{\Gamma}\times\mathbb{X}$.} For each $k\in\mathbb{N}$, let 
 $x^k$ be such that $\mathcal{B}(x^k-x^*)=X^k-X^*$ for $X^k=P^*{\rm BlkDiag}(\Lambda^{+}(X^*),\frac{1}{k}I_{|\beta^*\backslash\gamma_0^*|},0_{\widetilde{\gamma}^*}){P^*}^{\top}$,
 $w^k\!=w$ and $v^k\!=\mathcal{B}^*S^k+\mathcal{A}^*y^*$ with $S^k\!=\!P^*{\rm BlkDiag}(0_{\alpha^*\cup (\beta^*\setminus \gamma^*_0)},-\frac{1}{k}I_{|\gamma^*_0|},{\Lambda^{-}(S^*)}){P^*}^{\top}$. Fix any $k\in\mathbb{N}$. From $X^k\in\mathbb{S}_{+}^n$ and $X^*\in\mathbb{S}_{+}^n$, we have $X^k\!-\!X^*\in\mathcal{T}_{\mathbb{S}_{+}^n}(X^*)$. By the expressions of $X^k$ and $X^*$ and equation \eqref{SDCNor}, it is easy to check that $-X^*\in\mathcal{N}_{\mathbb{S}_{+}^n}(-S^*)$ and $-X^k\in\mathcal{N}_{\mathbb{S}_{+}^n}(-S^*)$, so $X^k\!-\!X^*\in{\rm Sp}(\mathcal{N}_{\mathbb{S}_{+}^n}(-S^*))$. Thus, $X^k\!-\!X^*\in\mathcal{T}_{\mathbb{S}_{+}^n}(X^*)\cap{\rm Sp}(\mathcal{N}_{\mathbb{S}_{+}^n}(-S^*))=K^*$. The given assumption $K^*\subset\mathcal{B}\mathbb{X}$ implies that $x^k$ is well defined. Combining $\mathcal{B}(x^k\!-\!x^*)=X^k\!-\!X^*$ with $X^*=\mathcal{B}x^*\!-\!B$ leads to $X^k=\mathcal{B}x^k\!-\!B=g(x^k)$. On the other hand, by the expressions of $X^k$ and $S^k$, we have $S^k\in\mathcal{N}_{\mathbb{S}_{+}^n}(X^k)$. Along with $v^k=\mathcal{B}^*S^k+\mathcal{A}^*y^*$, it follows that $((x^k,v^k),w^k)\in {\rm gph}\mathcal{N}_{\Gamma}\times\mathbb{X}$ for each $k\in\mathbb{N}$.  

 \medskip
 \noindent
 {\bf Step 2: to prove the existence $\mathcal{K}\subset\mathbb{N}$ such that $\lim_{\mathcal{K}\ni k\to\infty}(x^k,v^k,w^k)=(x^*,v^*,w)$}.
 We claim that the sequence $\{x^k\!-\!x^*\}_{k\in\mathbb{N}}$ is bounded. If not, $\lim_{k\to\infty}\|x^k\!-\!x^*\|=\infty$. Along with $\lim_{k\to\infty}X^k=X^*$ and $\mathcal{B}(x^k\!-\!x^*)=X^k\!-\!X^*$, we have
 $\lim_{k\to\infty}\mathcal{B}\frac{x^k-x^*}{\|x^k-x^*\|}=0$. According to the boundedness of $\{\frac{x^k-x^*}{\|x^k-x^*\|}\}_{k\in\mathbb{N}}$, there exists an index set $\mathcal{K}_0\subset\mathbb{N}$ such that $\lim_{\mathcal{K}_0\ni k\to\infty}\frac{x^k-x^*}{\|x^k-x^*\|}=z^*$ with $\|z^*\|=1$. Noting that $\frac{x^k-x^*}{\|x^k-x^*\|}\in\mathcal{B}^{-1}(K^*)$ and the set $\mathcal{B}^{-1}(K^*)$ is closed, we have $z^*\in\mathcal{B}^{-1}(K^*)$, which by the injectivity of $\mathcal{B}$ on the set $\mathcal{B}^{-1}(K^*)$ results in $z^*=0$, a contradiction to $\|z^*\|=1$. Thus, the claimed conclusion holds, i.e., the sequence $\{x^k\!-\!x^*\}_{k\in\mathbb{N}}$ is bounded. Consequently, there exists an infinite index set $\mathcal{K}$ such that $\{x^k\!-\!x^*\}_{k\in\mathcal{K}}$ is convergent and its limit, denoted by $\widetilde{z}^*$, satisfies $\mathcal{B}\widetilde{z}^*=0$. Obviously, $\widetilde{z}^*\in \mathcal{B}^{-1}(K^*)$. The injectivity of $\mathcal{B}$ on $\mathcal{B}^{-1}(K^*)$ implies $\widetilde{z}^*=0$, and then $\lim_{\mathcal{K}\ni k\to\infty}x^k=x^*$ follows. While from $\lim_{k\to\infty}S^k=S^*$, we get $\lim_{k\to\infty}v^k=\mathcal{B}^*S^*+\mathcal{A}^*y^*=v^*$. The proof of this step is finished. 
  
 \medskip
 \noindent
 {\bf Step 3: to prove that $\lim_{\mathcal{K}\ni k\to \infty} d^2\delta_{\Gamma}(x^k|v^k)(w^k)=0$}. For each $k\in\mathcal{K}$, define 
 \begin{equation}\label{indexk1}
 \alpha_k\!:=\!\{i\in[n]\,|\,\lambda_i(X^k)>0\},\ \gamma_k\!:=\!\{i\in[n]\,|\,\lambda_i(S^k)<0\}\ {\rm and}\ \beta_k\!:=[n]\backslash(\alpha_k\cup\gamma_k). 
 \end{equation}
 Clearly, $\alpha_k=\alpha^*\cup(\beta^*\backslash\gamma_0^*),\gamma_k=\widetilde{\gamma}^*$ and $\beta_k=\emptyset$. From \eqref{Wtilde-equa-d} and \eqref{SDCTan}, for each $k\in\mathcal{K}$,   
 \[
   \mathcal{B}w^k=\mathcal{B}w\in\mathcal{T}_{\mathbb{S}_{+}^n}(X^k)
   \ \ {\rm and}\ \ \langle v^k, w^k\rangle=\langle \mathcal{B}^*S^k\!+\mathcal{A}^*y^*, w\rangle=\langle S^k,\mathcal{B}w\rangle=0.
 \]
 Combining these two equalities with $\mathcal{A}w^k=\mathcal{A}w=0$ and Proposition \ref{Gamset-prop} (iii) yields that $w^k\in\mathcal{C}_{\Gamma}(x^k,v^k)$ for each $k\in\mathcal{K}$, so $w^k\in{\rm dom}\,d^2 \delta_{\Gamma}(x^k|v^k)$ by Proposition \ref{Gamset-prop} (iv), and
 \begin{align}\label{max-prob}
  d^2 \delta_{\Gamma}(x^k|v^k)(w^k)
 &=2\max_{y\in \mathbb{R}^m, S\in\mathcal{N}_{\mathbb{S}_{+}^n}(X^k)\atop v^k=\mathcal{B}^*S+\mathcal{A}^* y}\!\langle -S,(\mathcal{B}w^k)(X^k)^{\dagger}(\mathcal{B}w^k)\rangle.
 \end{align}
 Next we prove $\lim_{\mathcal{K}\ni k\to \infty} d^2\delta_{\Gamma}(x^k|v^k)(w^k)=0$ under conditions (A) and (B), respectively. 

 \noindent
 {\bf Case 1: the condition (A) holds.} Fix any $k\in\mathcal{K}$. We claim that any feasible point $(y,S)$ of the maximum problem \eqref{max-prob} satisfies $S=S^k$. Pick any feasible solution $(y,S)$ of \eqref{max-prob}. We have $S\in\mathcal{N}_{\mathbb{S}_{+}^n}(X^k)$ and $v^k=\mathcal{B}^*S+\mathcal{A}^*y$. Along with $v^k=\mathcal{B}^*S^k+\mathcal{A}^*y^*$, 
 \begin{equation}\label{yS-equa}
  \mathcal{A}^*(y-y^*)=\mathcal{B}^*(S^k-S). 
  \end{equation}
  By the expressions of $X^k$ and $X^*$, it is easy to check that $[\![X^k]\!]^{\perp}\subset[\![X^*]\!]^{\perp}$, so it holds 
 \begin{equation}\label{key-relation}
  \mathcal{N}_{\mathbb{S}_{+}^n}(X^k)=\mathbb{S}_{-}^n\cap[\![X^k]\!]^{\perp}\subset\mathbb{S}_{-}^n\cap[\![X^*]\!]^{\perp}=\mathcal{N}_{\mathbb{S}_{+}^n}(X^*). 
 \end{equation}
 Recall that $S^k\in\mathcal{N}_{\mathbb{S}_{+}^n}(X^k)$ and  $S\in\mathcal{N}_{\mathbb{S}_{+}^n}(X^k)$, so we have $S^k-S\in{\rm Sp}(\mathcal{N}_{\mathbb{S}_{+}^n}(X^*))$. Along with the above \eqref{yS-equa} and condition (A), we get $S=S^k$, and the claimed conclusion holds. Now, from equation \eqref{max-prob} and $w^k\equiv w$, for each $k\in\mathcal{K}$ it holds
 \begin{align*}
  d^2 \delta_{\Gamma}(x^k|v^k)(w^k)
  &=2\langle -S^k,(\mathcal{B}w^k)(X^k)^{\dagger}(\mathcal{B}w^k)\rangle\\
  &=2\langle -(P^*)^{\top}S^kP^*,\widehat{W}(P^*)^{\top}(X^k)^{\dagger}(P^*)\widehat{W}\rangle\\
  &=-2\sum_{j\in \gamma_k}\sum_{i\in \alpha_k}\frac{\lambda_{j}(S^k)}{\lambda_i(X^k)}(\widehat{W}_{ji})^2=0,
 \end{align*}
 where the third equality is due to \eqref{Wtilde-equa-d}, and the fourth is obtained by using $\gamma_k=\widetilde{\gamma}^*$ and $\alpha_k=\alpha^*\cup(\beta^*\backslash\gamma_0^*)$. The above equation implies that $\lim_{\mathcal{K}\ni k\to \infty}d^2 \delta_{\Gamma}(x^k|v^k)(w^k)=0$.

 \medskip
 \noindent
 {\bf Case 2: the condition (B) holds.} Fix any $k\in\mathcal{K}$. We claim that the optimal value of the maximum problem \eqref{max-prob} cannot be positive. If not, let $(\widetilde{y},\widetilde{S})$ be an optimal solution of \eqref{max-prob} with $d^2\delta_{\Gamma}(x^k|v^k)(w^k)=\varpi_k>0$. Obviously, for every $t>0$, $t\widetilde{S}\in\mathcal{N}_{\mathbb{S}_{+}^n}(X^k)$. Together with $S^k\in\mathcal{N}_{\mathbb{S}_{+}^n}(X^k)$, using \eqref{key-relation} leads to $t\widetilde{S}\!-\!S^k\in{\rm Sp}(\mathcal{N}_{\mathbb{S}_{+}^n}(X^*))$.  According to the condition (B), for every $t>0$, there exists $y_{t}\in\mathbb{R}^m$ such that $\mathcal{A}^*(y^*\!-\!y_{t})=\mathcal{B}^*(t\widetilde{S}\!-\!S^k)$ or $\mathcal{B}^*(t\widetilde{S})+\mathcal{A}^*y_{t}=\mathcal{A}^*y^*+\mathcal{B}^*S^k=v^k$, which along with $t\widetilde{S}\in\mathcal{N}_{\mathbb{S}_{+}^n}(X^k)$ means that $(y_{t},t\widetilde{S})$ is a feasible solution of \eqref{max-prob} with the objective value $t\varpi_k$. This implies that the optimal value of \eqref{max-prob} is $\infty$, a contradiction to $w^k\in{\rm dom}\,d^2 \delta_{\Gamma}(x^k|v^k)$ for each $k\in\mathcal{K}$. Thus, the claimed conclusion holds. Note that the optimal value of problem \eqref{max-prob} is nonnegative due to the convexity of $\Gamma$. Then, for each $k\in\mathcal{K}$, the optimal value of \eqref{max-prob} equals $0$. We get $\lim_{\mathcal{K}\ni k\to \infty}d^2 \delta_{\Gamma}(x^k|v^k)(w^k)=0$. 
  
 The above three steps show that $w\in\mathcal{W}$ by noting that $\mathcal{W}$ in \eqref{MW-set} can be written as
 \begin{align*}
 \mathcal{W}&=\Big\{w\in \mathbb{X}\ |\ \exists\,\mathcal{K}\subset\mathbb{N}\ {\rm and}\ \{((x^k,v^k),w^k)\}_{k\in\mathcal{K}}\subset{\rm gph}\,\mathcal{N}_{\Gamma}\times\mathbb{X}\ {\rm with}\nonumber\\
 &\qquad\qquad\quad \lim_{\mathcal{K}\ni k\to \infty}(x^k,v^k,w^k)= (x^*,v^*,w)\ {\rm s.t.}\ \lim_{k\to \infty} d^2\delta_{\Gamma}(x^k|v^k)(w^k)=0\Big\}.
 \end{align*}  
 By the arbitrariness of $w\in{\rm Ker}\mathcal{A}\cap\widetilde{\Upsilon}$, we have ${\rm Ker}\nabla^2\varphi(x^*)\cap{\rm Ker}\mathcal{A}\cap\widetilde{\Upsilon}=\{0\}$. Moreover, by following the proof of Corollary \ref{Scorollary1-ptilt}, it is easy to check that $\widetilde{\Upsilon}=\widetilde{\Upsilon}^*$. 
 \end{proof}
\begin{remark}\label{remark3-Ncond}
 {\bf(a)} Under the setting of this section, the condition (A) is implied by the constraint nondegeneracy of problem \eqref{pNSDP} at $x^*$ by noting that the latter is equivalent to the following implication
 \[
  \mathcal{A}^*\Delta y+\mathcal{B}^*\Delta S=0,\,\Delta S\in{\rm Sp}(\mathcal{N}_{\mathbb{S}_{+}^n}(X^*))\ \Longrightarrow (\Delta y,\Delta S)=(0,0),
 \]
 while the condition (B) is equivalent to ${\rm Ker}\mathcal{A}\subset\mathcal{B}^{-1}[{\rm lin}(\mathcal{T}_{\mathbb{S}_{+}^n}(X^*))]$, which has no implication relation with the constraint nondegeneracy of problem \eqref{pNSDP} at $x^*$. 

 \medskip
 \noindent
 {\bf(b)} From Step 3 in the proof of Theorem \ref{Ncond1-ptilt}, for problem \eqref{pNSDP} without the constraint $\mathcal{A}x=b$, the condition (A) reduces to ${\rm Ker}\,\mathcal{B}^*\cap{\rm Sp}(\mathcal{N}_{\mathbb{S}_{+}^n}(X^*))=\{0\}$, while the condition (B) can be replaced by this one by noting that $\mathcal{B}^*(t\widetilde{S})=tv^k=\mathcal{B}^*(tS^k)$. 
\end{remark} 

 Comparing with the definition of $\Upsilon$ in Theorem \ref{Scond1-ptilt}, we see that $\widetilde{\Upsilon}$ is a subset of $\Upsilon$ with the block diagonal orthogonal $Q$ specified as an identity matrix. This means that the necessary characterization is weaker than the converse conclusion of Theorem \ref{Scond1-ptilt}. When replacing $\widetilde{\Upsilon}$ with $\Upsilon$ in the conclusion of Theorem \ref{Ncond1-ptilt}, we meet a great challenge to construct the sequence $\{((x^k,v^k),w^k)\}_{k\in\mathbb{N}}\subset{\rm gph}\,\mathcal{N}_{\Gamma}\times\mathbb{X}$ satisfying \eqref{XGamWk-equa1}. The main reason to cause this difficulty is that a useful structure for $Q_p$, the last block of $Q$, cannot be achieved because the matrices $S^*$ and $S$ in the proof of Theorem \ref{Scond1-ptilt} have no relation. However, when replacing $\widetilde{\Upsilon}$ with the set $\widehat{\Upsilon}$ appearing in Theorem \ref{Scond2-ptilt}, this difficulty can be overcome under the conditions of Theorem \ref{Ncond1-ptilt}. This result is stated as follows. 

\begin{theorem}\label{Ncond2-ptilt}
 Consider a local optimal solution $x^*$ of \eqref{pNSDP}. Let $v^*\!:=-\nabla\varphi(x^*)$ and $X^*\!:=g(x^*)$. Suppose that $\mathcal{G}$ is metrically regular at $(x^*,(0,0))$. Pick any $(y^*,S^*)\in\mathcal{M}(x^*,v^*)$. Let $\alpha^*,\gamma^*$ and $\beta^*$ be defined by \eqref{alp-gamj}-\eqref{alp-beta} with $(X,S)=(X^*,S^*)$. Then, under the conditions of Theorem \ref{Ncond1-ptilt}, the tilt stability of $x^*$ implies ${\rm Ker}\nabla^2\varphi(x^*)\cap{\rm Ker}\mathcal{A}\cap\widehat{\Upsilon}=\{0\}$ or equivalently ${\rm Ker}\nabla^2\varphi(x^*)\cap{\rm Ker}\mathcal{A}\cap\widehat{\Upsilon}^{*}=\{0\}$.
 \end{theorem}
 \begin{proof}
 By Lemma \ref{Tiltlemma}, it suffices to prove  ${\rm Ker}\mathcal{A}\cap\widehat{\Upsilon}\subset\mathcal{W}$. Pick any $w\in{\rm Ker}\mathcal{A}\cap\widehat{\Upsilon}$. Then, $\mathcal{A}w=0$, and there exist  $\beta_{+}^*,\beta_{0}^*,\beta_{-}^*,\widetilde{\gamma}^*\subset[n]$ and matrices $P^*\!\in\mathbb{O}(X^*)\cap\mathbb{O}(S^*)$ and 
 $Q=\!{\rm BlkDiag}(Q_1,\ldots,Q_{p})$ with $Q_j\!\in\mathbb{O}^{|\alpha_j|}$ for $j\in[p]$ and $Q_p=\!{\rm BlkDiag}(Q_{p}^{1},\ldots,Q_{p}^{s})$ for $Q_{p}^{1}\in\mathbb{O}^{|\alpha_p|}$ and $Q_p^{j}\in\mathbb{O}^{|\gamma_j|}$ for $j=2,\ldots,s$ such that $\beta^*=\beta_{-}^*\cup\beta_{+}^*\cup\beta_0^*\cup\gamma^*_0$ with $\gamma^*_0\!:=\widetilde{\gamma}^*\backslash\gamma^*$, $R:=P^*Q\in\mathbb{O}(X^*)\cap\mathbb{O}(S^*)$, and 
\begin{equation}\label{Wtilde-Q}
 \widetilde{W}:=R^\top (\mathcal{B}w)R=\begin{pmatrix}
  \widetilde{W}_{11}& \widetilde{W}_{12} & \widetilde{W}_{13} & 0_{[\alpha^*\cup\beta_{+}^*]\widetilde{\gamma}^*}\\
  \widetilde{W}_{12}^{\top}&\widetilde{W}_{22} &0  &  0\\
  \widetilde{W}_{13}^{\top}& 0 & 0 & 0\\
  0_{\widetilde{\gamma}^*[\alpha^*\cup\beta_{+}^*]} & 0 & 0 & 0
  \end{pmatrix}
\end{equation}
 with $\widetilde{W}_{11}\in\mathbb{S}^{|\alpha^*\cup\beta_{+}^*|},\widetilde{W}_{12}\in \mathbb{R}^{|\alpha^*\cup\beta_{+}^*|\times |\beta_{0}^*|},\widetilde{W}_{13}\in \mathbb{R}^{|\alpha^*\cup\beta_{+}^*|\times |\beta_{-}^*|}$ and $\widetilde{W}_{22}\in\mathbb{S}^{|\beta_{0}^*|}_+$. From $R\in\mathbb{O}(X^*)\cap\mathbb{O}(S^*)$, the matrices $X^*$ and $S^*$ have the following spectral decompositions
 \begin{equation}\label{XS-starDecomp}
   X^*\!=R\begin{pmatrix}
   \Lambda^{+}(X^*) & 0\\
   0 & 0\\
  \end{pmatrix}R^{\top}\ {\rm and}\ 
  S^*\!=R\begin{pmatrix}
   0 & 0 \\
   0 & \Lambda^{-}(S^*)
 \end{pmatrix}R^{\top}
 \end{equation}
 with $\Lambda^{+}(X^*)={\rm Diag}(\lambda^{+}(X^*))$ and $\Lambda^{-}(S^*)={\rm Diag}(\lambda^{-}(S^*))$. For each $k\in\mathbb{N}$, let 
 $x^k$ be such that $\mathcal{B}(x^k\!-\!x^*)=X^k\!-\!X^*$ with $X^k\!=R{\rm BlkDiag}(\Lambda^{+}(X^*),\frac{1}{k}I_{|\beta^*\backslash\gamma_0^*|},0_{\widetilde{\gamma}^*})R^{\top}$, $w^k\!=w$ and $v^k\!=\mathcal{B}^*S^k+\mathcal{A}^*y^*$ with $S^k\!=R{\rm BlkDiag}(0_{\alpha^*\cup (\beta^*\setminus\gamma^*_0)},-\frac{1}{k}I_{|\gamma^*_0|},{\Lambda^{-}(S^*)})R^{\top}$. By using the same arguments as those for Theorem \ref{Ncond1-ptilt}, there exist an index set $\mathcal{K}\subset\mathbb{N}$ and a sequence $\{((x^k,v^k),w^k)\}_{k\in\mathcal{K}}\subset {\rm gph}\mathcal{N}_{\Gamma}\times\mathbb{X}$ with $\lim_{\mathcal{K}\ni k\to\infty}(x^k,v^k,w^k)=(x^*,v^*,w)$ such that $\lim_{\mathcal{K}\ni k\to \infty} d^2\delta_{\Gamma}(x^k|v^k)(w^k)=0$. Then, $w\in\mathcal{W}$, and the arbitrariness of $w\in{\rm Ker}\mathcal{A}\cap\widehat{\Upsilon}$ implies that the desired inclusion holds. 
 \end{proof}

 Theorem \ref{Ncond2-ptilt} along with Theorem \ref{Scond2-ptilt} provides a point-based sufficient and necessary characterization for the tilt stability of problem \eqref{pNSDP} with $g$ being the affine mapping, which is stated as follows.
 \begin{corollary}\label{SNcond0-ptilt}
 Consider a local optimal solution $x^*$ of problem \eqref{pNSDP} with $\nabla^2\varphi(x)\succeq 0$ on an open neighborhood $\mathcal{N}^*\!\subset\mathcal{O}$ of $x^*$ and $g(x)=\mathcal{B}(x)-B$. Let $v^*\!:=\!-\nabla\varphi(x^*)$ and $X^*\!:=g(x^*)$. Suppose that $\mathcal{G}$ is metrically regular at $(x^*,(0,0))$, and that $\mathcal{M}(x^*,v^*)\!=(\mathcal{A}^*)^{-1}(v^*\!-\nabla\!g(x^*)S^*)\times\{S^*\}$ for some $S^*\in\mathcal{N}_{\mathbb{S}_{+}^n}(X^*)$. Let $\alpha^*,\gamma^*$ and $\beta^*$ be defined by \eqref{alp-gamj}-\eqref{alp-beta} with $(X,S)=(X^*,S^*)$. Then, under the conditions of Theorem \ref{Ncond1-ptilt}, the solution $x^*$ is tilt-stable if and only if ${\rm Ker}\nabla^2\varphi(X^*)\cap\widehat{\Upsilon}=\{0\}$.
 \end{corollary}
  
When $\mathbb{X}=\mathbb{S}^n$, $\varphi$ is a twice continuously differentiable convex function, $\mathcal{A}=0$ and $g$ is an identity mapping, problem \eqref{pNSDP} becomes the following one considered in \cite[Example 3.4]{Nghia24}: 
\begin{equation}\label{special-pSDP1}
 \min_{X\in\mathbb{S}^n}\,\varphi(X)+\delta_{\mathbb{S}_{+}^n}(X).
\end{equation}
As mentioned in the introduction, the constraint nondegeneracy automatically holds for \eqref{special-pSDP1}. By Remark \ref{remark3-Ncond} (a), the assumptions of Theorem \ref{Ncond2-ptilt} hold, so the following sufficient and necessary characterization holds for the tilt stability of \eqref{special-pSDP1}. 
\begin{corollary}\label{SNcond1-ptilt}
 Consider a local optimal solution $X^*$ of problem \eqref{special-pSDP1} with $S^*=-\nabla\varphi(X^*)$. Let $\alpha^*,\gamma^*$ and $\beta^*$ be defined by \eqref{alp-gamj}-\eqref{alp-beta} with $(X,S)=(X^*,S^*)$. Then $x^*$ is a tilt-stable solution of \eqref{special-pSDP1} if and only if ${\rm Ker}\nabla^2\varphi(X^*)\cap\widehat{\Upsilon}=\{0\}$ or ${\rm Ker}\nabla^2\varphi(X^*)\cap\widehat{\Upsilon}^{*}=\{0\}$.
 \end{corollary}
 
 By comparing $\widehat{\Upsilon}^{*}$ with the set $\mathfrak{T}_{\mathbb{S}}$ in \cite[Example 3.15]{Nghia24}, it is easy to check $\widehat{\Upsilon}^{*}=\mathfrak{T}_{\mathbb{S}}$. Thus, our sufficient and necessary characterization in Corollary \ref{SNcond0-ptilt} for the tilt stability of problem \eqref{pNSDP} with an affine $g$ extends the one obtained in \cite{Nghia24} for a special convex SDP to a class of nonconvex SDPs.    
\subsection{Application to linear SDPs}\label{sec3.3}
 Consider the space $\mathbb{X}=\mathbb{S}^n$. When $\varphi(x)=\langle C,x\rangle$ for a given $C\in\mathbb{S}^n$, and $g(x)\equiv x$ for $x\in\mathbb{X}$, problem \eqref{pNSDP} reduces to the following standard linear SDP
 \begin{equation}\label{pSDP1}
  \min_{X\in\mathbb{S}^n}\big\{\langle C,X\rangle\ \ {\rm s.t.}\ \mathcal{A}X=b,\,X\in\mathbb{S}_{+}^n\big\},
 \end{equation}
 whose dual problem takes the following form 
 \begin{equation}\label{dSDP1}
  \max_{y\in\mathbb{R}^m,S\in\mathbb{S}^n}\big\{b^{\top}y\ \ {\rm s.t.}\ \mathcal{A}^*y+S=C,\,S\in\mathbb{S}_{+}^n\big\}. 
 \end{equation}
 Invoking the conclusions of Sections \ref{sec3.2}-\ref{sec3.3} yields the following result for tilt stability of \eqref{pSDP1}. 
\begin{proposition}\label{pSDP-prop1}
 Consider an optimal solution $X^*$ of the linear SDP \eqref{pSDP1}. Suppose that 
 \begin{equation*}
 \mathcal{G}(X)=\begin{pmatrix}
  \mathcal{A}X-b\\ X
 \end{pmatrix}-\begin{pmatrix}
   \{0\}^m\\ \mathbb{S}_{+}^n
   \end{pmatrix}\quad\ \forall X\in\mathbb{S}^n
 \end{equation*}
 is metrically regular at $(X^*,(0,0))$. Pick any $(y^*,S^*)\in\mathcal{M}(X^*,-C)$. Let $\alpha^*,\gamma^*$ and $\beta^*$ be defined by \eqref{alp-gamj}-\eqref{alp-beta} with $(X,S)=(X^*,S^*)$. The following two assertions hold true.
 \begin{itemize}
 \item [(i)] If $(y^*,S^*)\in\mathop{\arg\min}\limits_{(y,S)\in\mathcal{M}(X^*,-C)}{\rm rank}(S)$ and ${\rm Ker}\mathcal{A}\cap[\![C]\!]^{\perp}\cap\Upsilon^*\!=\{0\}$, the solution $X^*$ is tilt-stable. If ${\rm Sp}(\mathcal{N}_{\mathbb{S}_{+}^n}(X^*))\cap {\rm Im}\mathcal{A}^*\!=\!\{0\}$ or ${\rm Sp}(\mathcal{N}_{\mathbb{S}_{+}^n}(X^*))\subset\!\mathcal{A}^*\mathbb{R}^m$, the tilt stability of $X^*$ implies ${\rm Ker}\mathcal{A}\cap\widetilde{\Upsilon}^*=\{0\}$. 

 \item [(ii)] If $\mathcal{M}(X^*,-C)\!=(\mathcal{A}^*)^{-1}(-C\!-S^*)\times\{S^*\}$ and ${\rm Ker}\mathcal{A}\cap\widehat{\Upsilon}^{*}=\{0\}$, then $X^*$ is tilt-stable. If ${\rm Sp}(\mathcal{N}_{\mathbb{S}_{+}^n}(X^*))\cap {\rm Im}\mathcal{A}^*=\!\{0\}$ or ${\rm Sp}(\mathcal{N}_{\mathbb{S}_{+}^n}(X^*))\subset\!\mathcal{A}^*\mathbb{R}^m$, the tilt stability of $X^*$ means ${\rm Ker}\mathcal{A}\cap\widehat{\Upsilon}^{*}=\{0\}$.
 \end{itemize} 
 \end{proposition}
 \begin{remark}
 {\bf(a)} Checking ${\rm Ker}\mathcal{A}\cap[\![C]\!]^{\perp}\cap\Upsilon^*\!=\{0\}$ or ${\rm Ker}\,\mathcal{A}\cap\widehat{\Upsilon}^*=\{0\}$ is equivalent to verifying if the optimal value of a reverse convex problem is zero; for example, checking ${\rm Ker}\,\mathcal{A}\cap\Upsilon^*=\{0\}$ requires verifying if the optimal value of the following problem is zero: 
 \begin{equation*}
 \max_{A\in\mathbb{S}^{|\alpha^*|},D\in \mathbb{S}_{+}^{|\beta^*|}\atop E\in \mathbb{R}^{|\alpha^*|\times |\beta^*|}}\!\left\{\|A\|_F^2\!+\!\|E\|^2_F\!+\!\|D\|^2_F\ \ {\rm s.t.}\ \mathcal{A}W=0,W=\!\begin{pmatrix}
  A & E & 0\\
 E^{\top} & D & 0\\
  0 &0 & 0
 \end{pmatrix}\!\right\}.
\end{equation*}

 \noindent
 {\bf(b)} The condition ${\rm Ker}\mathcal{A}\cap\widehat{\Upsilon}^{*}=\{0\}$ is the same as the implication in \cite[Eq. (52)]{ChanSun08} which, under the uniqueness of the multiplier set, is equivalent to the SSOSC of \cite[Definition 13]{ChanSun08} by \cite[Lemma 14]{ChanSun08} and the constraint nondegeneracy at $(y^*,S^*)$ to the dual SDP \eqref{dSDP1} by \cite[Proposition 15]{ChanSun08}. Then,  Proposition \ref{pSDP-prop1} (ii) shows that under the uniqueness of the multiplier set, the tilt-stability of the linear SDP \eqref{pSDP1} is weaker than the SSOSC of \cite[Definition 13]{ChanSun08} as well as the constraint nondegeneracy at $(y^*,S^*)$ to \eqref{dSDP1}. When the multiplier set is not a singleton, there is no implication relation between the sufficient characterization ${\rm Ker}\mathcal{A}\cap[\![C]\!]^{\perp}\cap\Upsilon^*=\{0\}$ of Proposition \ref{pSDP-prop1} (i) and the constraint nondegeneracy at $(y^*,S^*)$ to \eqref{dSDP1} (i.e., ${\rm Ker}\mathcal{A}\cap{\rm Sp}(\mathcal{N}_{\mathbb{S}_{+}^n}(-S^*))=\{0\}$) because $P^*Q\in\mathbb{O}(X^*)\cap\mathbb{O}(S^*)$ does not necessarily hold, but the necessary characterization ${\rm Ker}\mathcal{A}\cap\widetilde{\Upsilon}^*=\{0\}$ is weaker than ${\rm Ker}\mathcal{A}\cap{\rm Sp}(\mathcal{N}_{\mathbb{S}_{+}^n}(-S^*))\!=\!\{0\}$.
\end{remark}  

 Next we consider that $\mathcal{A}=0,b=0$, and $\varphi(x)=\langle d,x\rangle$ and $g(x)=\mathcal{B}x-B$ for $x\in\mathbb{X}$, where $d\in\mathbb{X}$ and $B\in\mathbb{S}^n$ are the given data. Now problem \eqref{pNSDP} is specified as the following one 
 \begin{equation}\label{pSDP2}
  \min_{x\in\mathbb{X}}\big\{\langle d,x\rangle\ \ {\rm s.t.}\ \ \mathcal{B}x-B\in\mathbb{S}_{+}^n\big\},
 \end{equation}
 which corresponds to the dual SDP \eqref{dSDP1} with $\mathbb{X}=\mathbb{R}^m,b=-d,\mathcal{A}^*=-\mathcal{B}$ and $B=-C$. From Sections \ref{sec3.2}-\ref{sec3.3}, we have the following conclusion for the tilt stability of \eqref{pSDP2}. 
\begin{proposition}\label{pSDP-prop2}
 Consider an optimal solution $x^*$ of the linear SDP \eqref{pSDP2}. Suppose that 
 \begin{equation}\label{MHmap1}
 \mathcal{H}(x):=\mathcal{B}x-B-\mathbb{S}_{+}^n\quad{\rm for}\ x\in\mathbb{X}
 \end{equation}
 is metrically regular at $(x^*,0)$. Pick any $S^*\!\in\mathcal{M}(x^*,-d)$. Let $X^*\!=\mathcal{B}x^*\!-B$, and let $\alpha^*,\gamma^*$ and $\beta^*$ be the index sets defined by \eqref{alp-gamj}-\eqref{alp-beta} with $(X,S)=(X^*,S^*)$, and let $K^*:=\mathcal{T}_{\mathbb{S}_{+}^n}(X^*)\cap{\rm Sp}(\mathcal{N}_{\mathbb{S}_{+}^n}(-S^*))$. The following assertions hold true.
 \begin{itemize}
  \item [(i)] If $S^*\in\mathop{\arg\min}\limits_{S\in\mathcal{M}(x^*,-d)}{\rm rank}(S)$ and $[\![d]\!]^{\perp}\cap\Upsilon^*=\{0\}$, the solution $x^*$ is tilt-stable. If $K^*\subset\mathcal{B}\mathbb{X}$, $\mathcal{B}$ is injective on $\mathcal{B}^{-1}(K^*)$ and ${\rm Ker}\,\mathcal{B}^*\cap{\rm Sp}(\mathcal{N}_{\mathbb{S}_{+}^n}(X^*))=\{0\}$, the tilt stability of $x^*$ implies $[\![d]\!]^{\perp}\cap\widetilde{\Upsilon}^*=\{0\}$.

  \item [(ii)] If $\mathcal{M}(x^*,-d)=\{S^*\}$ and $\widehat{\Upsilon}^{*}\!=\!\{0\}$, the solution $x^*$ is tilt-stable. If $K^*\subset\mathcal{B}\mathbb{X}$, $\mathcal{B}$ is injective on the set $\mathcal{B}^{-1}(K^*)$ and ${\rm Ker}\,\mathcal{B}^*\cap{\rm Sp}(\mathcal{N}_{\mathbb{S}_{+}^n}(X^*))=\{0\}$, the tilt stability of $x^*$ implies $\widehat{\Upsilon}^{*}=\{0\}$.
 \end{itemize} 
 \end{proposition}   
 \begin{remark}
  Under the uniqueness of the set $\mathcal{M}(x^*,-d)$, the condition $\widehat{\Upsilon}^{*}=\{0\}$ is proved in \cite[Lemma 14]{ChanSun08} to be equivalent to the SSOSC of the linear SDP \eqref{pSDP2}, which by \cite[ Proposition 15]{ChanSun08} is also equivalent to the constraint nondegeneracy at $-S^*$ to the dual of \eqref{pSDP2}, i.e., $\mathcal{B}^*{\rm lin}(\mathcal{T}_{\mathbb{S}_{+}^n}(-S^*))=\mathbb{X}$. Note that the condition ${\rm Ker}\,\mathcal{B}^*\cap{\rm Sp}(\mathcal{N}_{\mathbb{S}_{+}^n}(X^*))=\{0\}$ is exactly the constraint nondegeneracy at $x^*$ to \eqref{pSDP2}. Thus, Proposition \ref{pSDP-prop2} (ii) shows that the tilt stability of \eqref{pSDP2} is much weaker than the SSOSC of the linear SDP \eqref{pSDP2} as well as the constraint nondegeneracy at $-S^*$ to its dual problem. 
 \end{remark} 
 \section{Conclusion}\label{sec4.0}

 For the nonlinear SDP problem \eqref{pNSDP} with a convex feasible set, we  derived point-based sufficient characterizations in Theorems \ref{Scond1-ptilt} and \ref{Scond2-ptilt} for its tilt-stable local optimal solutions by using the multiplier of the minimum rank and imposing a suitable restriction on the multiplier set, respectively. Then, for the linear PSD constraint set, we  established a point-based necessary characterization in Theorem \ref{Ncond1-ptilt} with a certain gap from the converse of Theorem \ref{Scond1-ptilt}, and a point-based necessary characterization in Theorem \ref{Ncond2-ptilt} without gap from the converse of Theorem \ref{Scond2-ptilt}. These results are applied to the standard linear primal and dual SDPs to provide the specific characterizations for their tilt-stable solutions, and their relations with the SSOSC and the dual constraint nondegeneray were also clarified. To the best of knowledge, this is the first work to establish point-based sufficient and/or necessary characterizations for true non-polyhedral conic programs without constraint nondegeneracy condition even without requiring the uniqueness of multipliers.

\end{document}